\documentclass[pdflatex,sn-mathphys-num]{sn-jnl}

\usepackage{graphicx}%
\usepackage{multirow}%
\usepackage{amsmath,amssymb,amsfonts,mathtools}%
\usepackage{amsthm}%
\usepackage{mathrsfs}%
\usepackage[title]{appendix}%
\usepackage{xcolor}%
\usepackage{textcomp}%
\usepackage{manyfoot}%
\usepackage{booktabs}%
\usepackage{algorithm}%
\usepackage{algorithmicx}%
\usepackage{algpseudocode}%
\usepackage{listings}%
\usepackage{bm,mathtools}%

\mathtoolsset{showonlyrefs}

\theoremstyle{thmstyleone}%
\newtheorem{theorem}{Theorem}[section]
\newtheorem{proposition}[theorem]{Proposition}
\newtheorem{lemma}[theorem]{Lemma}

\theoremstyle{thmstyletwo}
\newtheorem{remark}[theorem]{Remark}

\theoremstyle{thmstylethree}
\newtheorem{definition}[theorem]{Definition}

\newcommand{\R}{\mathbb{R}}
\newcommand{\T}{\mathbb{T}}

\newcommand{\dd}{\,\mathrm{d}}

\begin{document}

\title[Optimal error estimates for half-way bounce-back]{Optimal error estimates for the half-way bounce-back lattice Boltzmann method for the Stokes equations}

\author*{\fnm{Kai} \sur{Koike}}\email{kaikoike.math@keio.jp}

\affil{\orgdiv{Department of Mathematics, Hiyoshi Campus},
  \orgname{Keio University},
  \orgaddress{\city{Yokohama 223-8521}, \country{Japan}}}

\abstract{%
We give a mathematical proof of the optimal convergence rates for the D2Q9 BGK lattice Boltzmann method with the half-way bounce-back rule for the incompressible Stokes equations in a flat channel. The convergence rates are second-order for the velocity and first-order for the pressure as the lattice spacing $h$ tends to zero, in agreement with formal analyses and numerical experiments, whereas the available rigorous convergence theorems only yield an $O(h^{1/2})$ bound for the velocity error. A key step in the proof is a decomposition of the leading boundary consistency error into macroscopic and kinetic components. These components are absorbed by suitably constructed Stokes and discrete Knudsen layer correctors, respectively. Incorporating these correctors into the prediction function used in previous rigorous analyses, we obtain a refined prediction function with consistency errors of sufficiently high order. Combined with the known weighted $L^2$-stability estimate, this gives the optimal convergence rates.
}



\maketitle

\tableofcontents

\section{Introduction}\label{sec:intro}

The Lattice Boltzmann method (LBM) is a computational fluid dynamics method based on ideas from the kinetic theory of gases. The basic idea is to construct a discrete velocity kinetic model on lattice sites with simplified collision processes so that the macroscopic quantities approximate solutions to the target continuum fluid dynamics equations---which are the Stokes equations in this paper. Its simple implementation, scalability, and straightforward handling of complex geometries have attracted the attention of many practitioners. On the other hand, its indirect formulation makes its mathematical analysis less direct than that of conventional methods such as finite-difference or finite-volume schemes. The validity of the LBM relies on some discrete versions of the theory of fluid dynamic limit in the kinetic theory of gases; cf.~\cite{Sone2002}. The corresponding asymptotic theory for the LBM is usually carried out based on the Chapman--Enskog expansions. There is an extensive literature on formal asymptotic analyses for the LBM; we refer to the reviews~\cite{ChenDoolen1998,Succi2001,aidun2010lattice,LallemandLuoKrafczykYong2021} and the references therein. From a mathematical point of view, however, the use of Chapman--Enskog expansions as a rigorous approximation procedure involves delicate issues concerning regularity and truncation; see for example~\cite{SaintRaymond2013}. Rigorous convergence analyses therefore require quantitative control of approximation errors beyond the formal asymptotic expansions.

Despite these difficulties, rigorous convergence results for the LBM and closely related schemes have been obtained in several different settings. Elton, Levermore, and Rodrigue~\cite{EltonLevermoreRodrigue1995} obtained convergence theorems for LBM schemes to convection--diffusion equations using truncated Hilbert expansions. Dellacherie proved convergence theorems for LBM schemes targeting 1D convection--diffusion equations based on equivalence to finite-difference type schemes and discrete maximum principles~\cite{dellacherie2014construction}. For the finite-difference formulation of the LBM, we also refer to the work by Junk~\cite{Junk_finite} and Bellotti, Graille, and Massot~\cite{BellottiGrailleMassot2022,Bellotti2023Truncation}. Caetano, Dubois, and Graille~\cite{CaetanoDuboisGraille2024}, Aregba-Driollet~\cite{AregbaDriollet2024}, and Aregba-Driollet and Bellotti~\cite{AregbaDriolletBellotti2026} proved convergence of LBM schemes to entropy solutions of hyperbolic conservation laws. These results skillfully use structures specific to their respective regimes, such as equivalence to finite-difference schemes, discrete maximum principles, monotonicity, or entropy inequalities. For incompressible flow, Junk and Yong proved convergence theorems for a continuous-in-space LBM scheme for the Navier--Stokes equations~\cite{JunkYong2003}. Subsequently, Junk and Yang~\cite{JunkYang2008,JunkYang2009} showed convergence theorems for LBM schemes for the Stokes and the Navier--Stokes equations based on the asymptotic analyses by Junk, Klar, and Luo~\cite{junk2005asymptotic}, Junk and Yang~\cite{JunkYang2005}, and the $L^2$-stability estimate of Junk and Yong~\cite{JunkYong2009}. These works for incompressible flows utilize truncated Hilbert expansions as in~\cite{EltonLevermoreRodrigue1995} and are most directly relevant to our analysis in this paper.

A central ingredient in the analysis for incompressible flows cited above is the \emph{prediction function}. Starting from the solution $(u,p)$ to the target Stokes or Navier--Stokes equations, based on a formal Hilbert-type expansion analyzed in~\cite{junk2005asymptotic,JunkYang2005}, they construct a distribution function
\[
    \hat f_h
    =
    \sum_{k=0}^{5}h^k f^{(k)},
\]
where $h$ is the lattice spacing. The prediction function can be inserted into the LBM scheme, so that its consistency error in the interior and on the boundary can be estimated by Taylor expansion. Combining these consistency estimates with the weighted $L^2$-stability of~\cite{JunkYong2009}, Junk and Yang~\cite{JunkYang2008,JunkYang2009} obtained rigorous convergence results both on periodic domains and, with the bounce-back rule, on bounded domains. In the periodic case, their estimate gives the optimal second-order convergence for the velocity and first-order convergence for the pressure.

The situation is considerably more subtle in the presence of a boundary. The bounce-back rule is one of the simplest and most widely used boundary treatments for the LBM. When the physical wall is located half a lattice spacing away from the nearest layer of lattice sites---the so-called half-way configuration---formal asymptotic analyses and numerical experiments predict second-order convergence for the velocity and first-order convergence for the pressure to solutions to Stokes or Navier--Stokes equations with the homogeneous Dirichlet boundary conditions; see~\cite{ChenDoolen1998,Succi2001,aidun2010lattice,LallemandLuoKrafczykYong2021} and~\cite{JunkYang2005}. By contrast, the rigorous convergence estimates of Junk and Yang~\cite{JunkYang2008,JunkYang2009} yield only an $O(h^{1/2})$ velocity error in the half-way configuration, and the convergence is not proved for the pressure. The origin of this loss is the fact that the consistency error of the prediction function $\hat{f}_h$ for the bounce-back rule is $O(h^3)$; after taking the weighted $L^2$-norm on the boundary and summing over $O(h^{-2})$ time steps (due to diffusive time scaling), one obtains only an $O(h^{3/2})$ bound for the $L^2$-norm of the difference between the prediction function $\hat{f}_h$ and the actual LBM solution. And after taking the velocity and pressure moments, they become $O(h^{1/2})$ and $O(h^{-1/2})$, respectively. Junk and Yang already observed that these estimates are likely to be
too coarse~\cite{JunkYang2008,JunkYang2009}.

Concerning the issue above on boundary errors, the work of Wei\ss~\cite{weiss2006numerical} is particularly relevant to us. He studied a one-dimensional Goldstein--Taylor-type LBM scheme targeting the heat equation on a bounded interval. For several boundary conditions, he established stability and second-order convergence of the macroscopic density. This provides an early example in which a sharp convergence rate on a bounded domain is obtained for a LBM scheme. An important ingredient of his analysis is the introduction of correctors designed to remove the leading term in the boundary consistency error~\cite[Section~5]{weiss2006numerical}.

The purpose of the present paper is to rigorously obtain the optimal convergence rates in a canonical setting in which the half-way bounce-back rule is expected to exhibit second-order velocity and first-order pressure convergence. Namely, we consider the linear D2Q9 BGK lattice Boltzmann method in the periodic flat channel
\[
    \Omega=\mathbb T\times(0,1)
\]
with the bounce-back rule under the diffusive scaling. The physical walls are placed half a lattice spacing away from the first and last layers of lattice sites. Under suitable smoothness and initialization assumptions, we prove that, for every fixed $T>0$,
\[
    \sup_{0\leq nh^2\leq T}
    \|U_h^n-u(nh^2)\|_{\ell_h^2}
    \leq
    C_T h^2,
    \qquad
    \sup_{0\leq nh^2\leq T}
    \|P_h^n-p(nh^2)\|_{\ell_h^2}
    \leq
    C_T h,
\]
where $(U_{h}^{n},P_{h}^{n})$ are the velocity and pressure computed from the LBM scheme, $(u,p)$ is the solution to the Stokes equations in $\Omega$ with the homogeneous Dirichlet boundary conditions, and $\| \cdot \|_{\ell_{h}^{2}}$ is the $L^2$-norm for lattice functions. The constant $C_T>0$ is independent of the lattice spacing $h>0$. To the best of our knowledge, this is the first rigorous estimate recovering the formally and numerically predicted second-order velocity accuracy and first-order pressure accuracy of the half-way bounce-back rule LBM for incompressible flows in a bounded domain.

The key point of the proof is a refinement of the prediction function $\hat{f}_h$ by adding some correctors based on a decomposition into macroscopic and kinetic components of the boundary consistency error. The macroscopic part is absorbed into correctors constructed from Stokes equations in the spirit of the technique of Dong, Ying, and Zhang for a MAC scheme for Stokes equations with Dirichlet boundary conditions~\cite{DongYingZhang2020}. The remaining kinetic part is absorbed into Knudsen layer correctors. The idea of introducing suitably constructed correctors to remove unwanted terms in the boundary error can be seen in the work of Wei\ss~\cite{weiss2006numerical}. However, in his setting, the construction of the correctors is purely algebraic. On the other hand, in the present multidimensional incompressible problem, the construction of the correctors requires solving Stokes equations with carefully chosen boundary data and Knudsen layer equations and is substantially more involved. After introducing the correctors, the resulting refined prediction function has a pointwise boundary residual of $O(h^5)$ and a pointwise interior residual of the form
\[
    O(h^6)
    +
    O\left(h^5\Gamma^{d}\right)
\]
for some $0 \leq \Gamma <1$, where $d$ is the distance to the closest wall in lattice units. Combining these bounds with the $L^2$-stability estimate of~\cite{JunkYong2009}, together with a suitable initialization, we obtain an $O(h^3)$ bound for the $L^2$-norm of the difference between the refined prediction function and the actual LBM solution. After taking the velocity and pressure moments, the optimal $O(h^2)$ velocity error and $O(h)$ pressure error follow.

Boundary layers of Knudsen type are, of course, classical in the kinetic theory of gases. For the Boltzmann equation, Knudsen layers form an essential part of the derivation of slip and no-slip boundary conditions for fluid dynamical equations. We refer to the books by Sone~\cite{Sone2002,Sone2007} for a systematic treatment. Knudsen layers in lattice kinetic models have also been analyzed; we refer for example to~\cite{CornubertDHumieresLevermore1991,GinzburgRouxSilva2015}. The novelty of the present work is therefore not the occurrence of a Knudsen layer itself, but its incorporation, together with the outer Stokes correctors, into a rigorous error analysis that
recovers the sharp convergence rate for the macroscopic quantities.

The paper is organized as follows. In Section~\ref{sec:formulation}, we formulate the LBM scheme considered in this paper, recall the convergence and stability results in~\cite{JunkYang2008,JunkYong2009}, and state our main theorem. Section~\ref{sec:correctors} is devoted to the construction of the Stokes and Knudsen layer correctors and the refined prediction function. The required interior and boundary residual estimates are established here. In Section~\ref{sec:convergence}, we combine these estimates with the $L^2$-stability estimate of~\cite{JunkYong2009} to prove the main
theorem. Section~\ref{sec:conclusion} contains concluding remarks and possible extensions.

\section{Formulation, known results, and the main theorem}
\label{sec:formulation}
\subsection{Geometry, D2Q9 velocity set, and notation}\label{sec:geometry}
We consider the channel
\begin{equation}
  \Omega=\T\times(0,1),
\end{equation}
where $\T \coloneqq \R / \mathbb{Z}$ is the 1D torus. We note that the boundary of $\Omega$ is
\begin{equation}
    \partial \Omega=\T \times \{ 0, 1 \}.
\end{equation}
Let $M \geq 2$ be a fixed integer and set $h=M^{-1}$. The domain $\Omega$ is discretized by the square lattice
\begin{equation}
    \Omega_h
    \coloneqq
    \left\{z_{j,\ell}=\left(jh,\left(\ell+\frac12\right)h\right) :
    j\in\mathbb Z/M\mathbb Z,\ \ell=0,\ldots,M-1\right\}.
\end{equation}
Correspondingly, we write
\begin{equation}
    \T_h \coloneqq \{ jh : j \in \mathbb{Z}/M\mathbb{Z} \}.
\end{equation}
Thus the physical walls are located half a lattice spacing away from the first and the last layers of lattice sites, namely, $z_{j,\ell}$ with $\ell=0,M-1$; this is the half-way configuration considered throughout the paper.

We adopt the D2Q9 velocity set, namely, the following 9 velocities:
\begin{equation}
\begin{aligned}
    &c_0=(0,0)^T,\quad c_1=(1,0)^T,\quad c_2=(0,1)^T,\quad c_3=(-1,0)^T,\quad c_4=(0,-1)^T,\\
    &c_5=(1,1)^T,\quad c_6=(-1,1)^T,\quad c_7=(-1,-1)^T,\quad c_8=(1,-1)^T.
\end{aligned}
\end{equation}
For readability, we often use the directional notation
\[
    E=1,\quad N=2,\quad W=3,\quad S=4,\quad
    NE=5,\quad NW=6,\quad SW=7,\quad SE=8.
\]
For example, $c_E=c_1=(1,0)^T$. Next, for $i \in \{ 0,\ldots,8 \}$, we denote by $i^*$ the index such that $c_{i^*}=-c_i$. For example, $1^*=3$ or equivalently $E^*=W$. To each velocity, we associate the weights 
\begin{equation}
    w_0=\frac49,\qquad
    w_1=\cdots=w_4=\frac19,\qquad
    w_5=\cdots=w_8=\frac1{36}.
\end{equation}

Let $T>0$ and define
\begin{equation}
    N_h
    \coloneqq
    \max \{ n \in \mathbb{N}_0 : nh^2 \leq T \}.
\end{equation}
For an integer $n$ with $0 \leq n \leq N_h$, we set $t_n=nh^2$.

For $\rho_0 \in \R$ and $j_0 \in \R^2$, we define the equilibrium distribution with density $\rho_0$ and momentum $j_0$ as the vector $F^{\mathrm{eq}}(\rho_0,j_0)=(F_{i}^{\mathrm{eq}}(\rho_0,j_0))_{i=0}^{8} \in \R^9$ with components
\begin{equation}
    F_i^{\mathrm{eq}}(\rho_0,j_0) \coloneqq w_i (\rho_0+3c_i \cdot j_0), \qquad i \in \{ 0,\ldots,8 \}.
\end{equation}
In addition, for a vector $f=(f_i)_{i=0}^{8} \in \R^9$, we define its density and momentum by
\begin{equation}\label{eq:rho_j}
    \rho_f \coloneqq \sum_{i=0}^{8}f_i, \qquad j_f \coloneqq \sum_{i=0}^{8}c_i f_i.
\end{equation}
Then the equilibrium projection of $f$ is defined as the vector
\begin{equation}
    \Pi f
    \coloneqq
    F^{\mathrm{eq}}(\rho_f,j_f) \in \R^9.
\end{equation}
For $\tau>0$, the BGK relaxation operator $J \colon \R^9 \to \R^9$ is defined by
\begin{equation}
    Jf \coloneqq \frac{1}{\tau}(\Pi f-f)=\frac{1}{\tau}\left( F^{\mathrm{eq}}(\rho_f,j_f)-f \right).
\end{equation}
Let $I \colon \R^9 \to \R^9$ be the identity operator and set
\begin{equation}\label{eq:def-gamma}
    \gamma
    \coloneqq
    1-\frac{1}{\tau}.
\end{equation}
The operator $\mathscr C \colon \R^9 \to \R^9$ is defined by
\begin{equation}\label{eq:def-collision-operator}
    \mathscr C
    \coloneqq
    I+J=\Pi+\gamma(I-\Pi).
\end{equation}

For $f \in \R^9$, we write
\begin{equation}\label{def:f-norm}
    \| f \|
    \coloneqq
    \sqrt{\sum_{i=0}^{8}w_{i}^{-1}f_{i}^{2}}.
\end{equation}
This weighted norm is of course equivalent to the standard Euclidean norm. We introduce the function space
\begin{equation}
    X_h
    \coloneqq
    \{ f \colon \Omega_h \to \R^9 \}
\end{equation}
endowed with the norm
\begin{equation}\label{def:norm-Xh}
    \| f \|_{X_h}
    \coloneqq
    \sqrt{\sum_{z \in \Omega_h}h^2 \| f(z) \|^2}.
\end{equation}
For $\varphi \colon \Omega \to \R^d$ or $\varphi \colon \Omega_h \to \R^d$ with $d=1$ or $2$, we write
\begin{equation}\label{eq:lh2}
    \| \varphi \|_{\ell_{h}^{2}} = \sqrt{\sum_{z \in \Omega_h} h^2 |\varphi(z)|^2},
\end{equation}
where $|\varphi(z)|$ is the standard Euclidean norm of $\R^d$.

\subsection{LBM scheme with the bounce-back rule}\label{sec:scheme}
We now introduce a LBM scheme that determines a sequence of $\R^9$-valued functions
\begin{equation}
    f_h^n \colon \Omega_h \to \R^9, \qquad 0 \leq n \leq N_h
\end{equation}
given initial data $f_{h}^{0} \colon \Omega_h \to \R^9$. In what follows, components of $f_{h}^{n}$ are written as
\begin{equation}
    f_{h}^{n}=(f_{h,i}^{n})_{i=0}^8.
\end{equation}
The update rules of the LBM scheme are as follows: for $0 \leq n \leq N_h-1$, $z \in \Omega_h$, and $i \in \{ 0,\ldots,8 \}$, we impose
\begin{equation}\label{eq:LBM}
    f_{h,i}^{n+1}(z+c_i h) = (\mathscr C f_{h}^{n}(z))_i \qquad \text{if $z+c_i h \in \Omega_h$}
\end{equation}
and
\begin{equation}\label{eq:bounce-back}
    f_{h,i^*}^{n+1}(z) = (\mathscr C f_{h}^{n}(z))_i \qquad \text{if $z+c_i h \notin \Omega_h$}.
\end{equation}
The second equation~\eqref{eq:bounce-back} is the bounce-back rule. Written out explicitly,~\eqref{eq:LBM} and~\eqref{eq:bounce-back} are
\begin{equation}
    f_{h,i}^{n+1}(z+c_i h) = f_{h,i}^{n}(z)+\frac{1}{\tau}\left( F_i^{\mathrm{eq}}(\rho_{f_{h}^{n}(z)},j_{f_{h}^{n}(z)})-f_{h,i}^{n}(z) \right) \qquad \text{if $z+c_i h \in \Omega_h$}
\end{equation}
and
\begin{equation}
    f_{h,i^*}^{n+1}(z) = f_{h,i}^{n}(z)+\frac{1}{\tau}\left( F_i^{\mathrm{eq}}(\rho_{f_{h}^{n}(z)},j_{f_{h}^{n}(z)})-f_{h,i}^{n}(z) \right) \qquad \text{if $z+c_i h \notin \Omega_h$},
\end{equation}
respectively. Defining the streaming operator $\widetilde S \colon X_h \to X_h$ by
\begin{equation}\label{eq:streaming}
    (\widetilde{S}f)_i(z)=
    \begin{dcases}
        f_i(z+c_i h) & (\text{if $z+c_i h \in \Omega_h$}), \\
        f_{i^*}(z) & (\text{if $z+c_i h \notin \Omega_h$})
    \end{dcases}
\end{equation}
for $f \in X_h$, the LBM scheme~\eqref{eq:LBM} and~\eqref{eq:bounce-back} can be written in the compact form
\begin{equation}\label{eq:LBM-compact}
    \widetilde{S}f_{h}^{n+1}=\mathscr C f_{h}^{n} \qquad (0 \leq n \leq N_h-1).
\end{equation}

\subsection{Known results: $L^2$-stability and a convergence theorem}
Let us review some results in~\cite{JunkYang2008,JunkYong2009} that are particularly relevant for us. Let $f_{h}^{n}$ be the solution to the LBM scheme~\eqref{eq:LBM-compact}. Define the associated macroscopic velocity and the pressure by
\begin{equation}\label{eq:macro-intro}
    U_h^n \coloneqq \frac{1}{h}j_{f_{h}^{n}}=\frac1h\sum_{i=0}^8c_i f_{h,i}^n,
    \qquad
    P_h^n \coloneqq \frac{1}{3h^2}\left( \rho_{f_{h}^{n}}-1 \right)=\frac1{3h^2}\left(\sum_{i=0}^8 f_{h,i}^n-1\right).
\end{equation}
Set the viscosity coefficient $\nu$ to
\begin{equation}
    \nu=\frac{\tau}{3}-\frac{1}{6}.
\end{equation}
Note that $\nu>0$ is equivalent to $\tau>1/2$. Let $\psi \colon \Omega \to \R^2$ be a divergence-free vector field satisfying $\psi=0$ on $\partial \Omega$. Upon taking the initial data $f_{h}^{0}$ appropriately depending on $\psi$, Junk and Yang proved in~\cite{JunkYang2008} that $(U_{h}^{n},P_{h}^{n})$ converges as $h \to 0$ to the solution to the Stokes equations
\begin{equation}\label{eq:stokes-intro}
\begin{dcases}
    \partial_tu+\nabla p=\nu\Delta u &\text{in $(0,T)\times\Omega$},\\
    \nabla\cdot u=0 &\text{in $(0,T)\times\Omega$},\\
    u=0 &\text{on $(0,T)\times\partial\Omega$}, \\
    u(0,\cdot)=\psi &\text{in $\Omega$}
\end{dcases}
\end{equation}
with periodicity in $x$ and the normalization $\int_\Omega p(t,x,y)\dd x \dd y=0$. A more precise statement is described below. The proof is based on an $L^2$-stability estimate and consistency analysis for the prediction function detailed below. First, we recall the $L^2$-stability estimate.

\begin{proposition}[{\cite[Lemma~2.2]{JunkYong2009}}]\label{prop:stability}
    We have
    \begin{equation}
        \| \widetilde S \|_{X_h \to X_h}=1
    \end{equation}
    and when $\tau \geq 1/2$, it holds that
    \begin{equation}
        \| \mathscr C \|_{X_h \to X_h} \leq 1.
    \end{equation}
\end{proposition}

In the remainder of the paper, we assume that
\begin{equation}
    \tau > \frac{1}{2}.
\end{equation}
As noted above, this corresponds to the positivity of the viscosity coefficient $\nu$. Constants appearing in our bounds may blow up as $\tau$ approaches $1/2$. For brevity, we fix $\tau>1/2$ and shall not consider the dependence of constants on $\tau$ hereafter.

We next recall the construction of the prediction function
\begin{equation}
    \hat{f}_{h} \colon [0,T] \times \Omega \to \R^9
\end{equation}
that is crucial in their proof. Note that the prediction function $\hat{f}_h$ is a function defined on the physical domain $\Omega$ and not on the lattice sites $\Omega_h$. First, let $(w,q)$ solve the modified Stokes equations
\begin{equation}\label{eq:Junk_higher_Stokes}
    \begin{dcases}
        \partial_t w+\nabla q=\nu \Delta w-\frac{\tau}{3}\nabla p & \text{in $(0,T) \times \Omega$}, \\
        \nabla \cdot w=-\partial_t p & \text{in $(0,T) \times \Omega$}, \\
        w=0 & \text{on $(0,T) \times \partial \Omega$}, \\
        w(0,\cdot)=w_{\mathrm{in}} & \text{in $\Omega$}
    \end{dcases}
\end{equation}
with periodicity in $x$ and the normalization $\int_{\Omega}q(t,x,y) \dd x \dd y=0$. The initial data $w_{\mathrm{in}}$ is required to satisfy the standard compatibility conditions of sufficiently high order. Define $f^{(k)} \colon [0,T] \times \Omega \to \R^9$ for $k=0,\ldots,5$ by
\begin{align}
    \label{eq:Junk_prediction}
    \begin{aligned}
    f^{(0)}
    &= F^{\mathrm{eq}}(1,0), \\
    f^{(1)}
    &= F^{\mathrm{eq}}(0,u), \\
    f^{(2)}
    &= F^{\mathrm{eq}}(3p,0)
    - \tau (V\cdot\nabla)f^{(1)}, \\
    f^{(3)}
    &= F^{\mathrm{eq}}(0,w)
    - \tau
    \left(
        \partial_t f^{(1)}
        +(V\cdot\nabla)f^{(2)}
        +\frac12 (V\cdot\nabla)^2 f^{(1)}
    \right), \\
    f^{(4)}
    &= F^{\mathrm{eq}}(3q,0)
    - \tau
    \left(
        \partial_t f^{(2)}
        +(V\cdot\nabla)f^{(3)}
        +\frac12 (V\cdot\nabla)^2 f^{(2)}
        +D_3f^{(1)}
    \right), \\
    f^{(5)}
    &= -\tau
    \left(
        \partial_t f^{(3)}
        +(V\cdot\nabla)f^{(4)}
        +\frac12 (V\cdot\nabla)^2 f^{(3)}
        +D_4f^{(1)}
        +D_3f^{(2)}
    \right),
    \end{aligned}
\end{align}
where
\begin{equation}
    (V\cdot\nabla)^l f
    \coloneqq
    \left(
        (c_0\cdot\nabla)^l f_0,
        \ldots,
        (c_8\cdot\nabla)^l f_8
    \right)^{T}
\end{equation}
and
\begin{equation}
    D_l
    \coloneqq
    \sum_{2a+b=l}
    \frac{\partial_t^a (V\cdot\nabla)^b}{a!\,b!}
\end{equation}
for $l \in \mathbb{N}$. Then the prediction function $\hat{f}_h \colon [0,T] \times \Omega \to \R^9$ is defined by
\begin{equation}\label{eq:Ph}
    \hat{f}_h \coloneqq \sum_{k=0}^{5}h^k f^{(k)}.
\end{equation}
The density and the momentum of $\hat{f}_h=(\hat{f}_{h,i})_{i=0}^{8}$ are related with $(u,p)$ and $(w,q)$ by the formulas
\begin{equation}\label{eq:Ph_moments}
    \sum_{i=0}^{8}\hat{f}_{h,i}=1+3h^2p+3h^4q, \qquad
    \sum_{i=0}^{8}c_i \hat{f}_{h,i}=hu+h^3w.
\end{equation}
See~\cite[p.~1487]{JunkYang2008}.\footnote{There is a missing factor of $3$ in their formula.}

The asymptotic analysis in~\cite{JunkYang2005} shows that the prediction function $\hat{f}_h$ satisfies the LBM scheme~\eqref{eq:LBM-compact} up to some order in $h$. More precisely, we have the following proposition.

\begin{proposition}[{\cite[Theorem~3]{JunkYang2008}}]\label{prop:Junk_error}
    Assume that $\tau>1/2$ and that the solution $(u,p)$ to~\eqref{eq:stokes-intro} is $C^5$-regular in $t$ and $x$. Then there exists a constant $C_T>0$ depending only on $T$ such that the prediction function $\hat{f}_h$ defined by~\eqref{eq:Junk_prediction} and~\eqref{eq:Ph} satisfies
    \begin{equation}\label{eq:Junk_residual}
        \widetilde{S}\hat{f}_{h}(t_{n+1})=\mathscr C\hat{f}_{h}(t_n)+\hat{r}_{h}^{n}+\widehat{R}_{h}^{n}
    \end{equation}
    for some $\hat{r}_{h}^{n},\widehat{R}_{h}^{n} \in X_h$ satisfying for $z \in \Omega_h$ and $0 \leq n \leq N_h-1$,
    \begin{equation}\label{eq:Junk_residual_order}
        \| \hat{r}_{h}^{n}(z) \| \leq C_T h^6, \qquad \| \widehat{R}_{h}^{n}(z) \| \leq C_T h^3,
    \end{equation}
    and
    \begin{equation}\label{eq:R_support}
        \widehat{R}_{h,i}^{n}(z)=0 \qquad (\text{for $i \in \{ 0,\ldots,8 \}$ with $z+c_i h \in \Omega_h$}).
    \end{equation}
     In particular, taking $C_T$ larger if necessary,\footnote{In what follows, by abuse of notation, we shall use the symbol $C_T$ to denote a generic positive constant depending only on $T$ that could possibly vary from place to place.}
    \begin{equation}\label{eq:Junk_residual_order_sum}
        \| \hat{r}_{h}^{n} \|_{X_h} \leq C_T h^6, \qquad \| \widehat{R}_{h}^{n} \|_{X_h} \leq C_T h^{7/2}.
    \end{equation}
\end{proposition}

The functions $\hat{r}_{h}^{n}$ and $\widehat{R}_{h}^{n}$ give the consistency error of the prediction function $\hat{f}_h$ to the LBM scheme~\eqref{eq:LBM-compact} at the interior and the boundary, respectively. In particular, the second inequality in~\eqref{eq:Junk_residual_order} means that the bounce-back rule~\eqref{eq:bounce-back} is satisfied up to $O(h^3)$.

Combining the $L^2$-stability estimate in Proposition~\ref{prop:stability} and the consistency analysis in Proposition~\ref{prop:Junk_error}, they showed the following convergence theorem. We shall also recall their proof as it is important to understand the necessary refinement of the prediction function in this paper.

\begin{theorem}[{\cite[Theorem~5]{JunkYang2008}}]\label{thm:Junk_convergence}
    Assume that $\tau>1/2$ and that the solution $(u,p)$ to~\eqref{eq:stokes-intro} is $C^5$-regular in $t$ and $x$. Let $f_{h}^{n}$ be the solution to the LBM scheme~\eqref{eq:LBM-compact} and let $\hat{f}_h$ be the prediction function defined by~\eqref{eq:Junk_prediction} and~\eqref{eq:Ph}. Then there exists $C_T>0$ depending only on $T$ such that
    \begin{equation}\label{eq:thm_Junk_f}
        \| f_{h}^{n}-\hat{f}_h(t_n) \|_{X_h} \leq C_T\left( \| f_{h}^{0}-\hat{f}_h(0) \|_{X_h}+h^{3/2} \right)
    \end{equation}
    for $0 \leq n \leq N_h$. In particular, $(U_{h}^{n},P_{h}^{n})$ defined by~\eqref{eq:macro-intro} satisfies
    \begin{equation}\label{eq:thm_Junk_U}
        \| U_{h}^{n}-u(t_n) \|_{\ell_{h}^{2}} \leq C_T\left( \frac{1}{h}\| f_{h}^{0}-\hat{f}_h(0) \|_{X_h}+h^{1/2} \right)
    \end{equation}
    and
    \begin{equation}\label{eq:thm_Junk_P}
        \| P_{h}^{n}-p(t_n) \|_{\ell_{h}^{2}} \leq C_T\left( \frac{1}{h^2}\| f_{h}^{0}-\hat{f}_h(0) \|_{X_h}+h^{-1/2} \right).
    \end{equation}
\end{theorem}

\begin{proof}
    Define $e_{h}^{n} \in X_h$ for $0 \leq n \leq N_h$ by
    \begin{equation}
        e_{h}^{n}
        \coloneqq
        f_{h}^{n}-\hat{f}_{h}(t_n).
    \end{equation}
    Set
    \begin{equation}
        A_h
        \coloneqq
        \widetilde{S}^{-1} \mathscr C.
    \end{equation}
    By Proposition~\ref{prop:stability}, we have
    \begin{equation}\label{eq:stability}
        \| \widetilde{S}^{-1} \|_{X_h \to X_h}=1,
        \qquad
        \| A_h \|_{X_h \to X_h} \leq 1.
    \end{equation}
    By~\eqref{eq:LBM-compact} and~\eqref{eq:Junk_residual}, we have
    \begin{align}
        f_{h}^{n+1}
        & =A_h f_{h}^{n}, \\
        \hat{f}_{h}(t_{n+1})
        & =A_h \hat{f}_{h}(t_n)
        +
        \widetilde S^{-1} \hat{r}_{h}^{n}
        +
        \widetilde S^{-1} \widehat{R}_{h}^{n}.
    \end{align}
    Subtracting the second equation from the first, we obtain
    \begin{equation}
        e_{h}^{n+1}=A_h e_{h}^{n}
        -
        \widetilde S^{-1} \hat{r}_{h}^{n}
        -
        \widetilde S^{-1} \widehat{R}_{h}^{n}.
    \end{equation}
    Hence by Duhamel's principle,
    \begin{equation}\label{eq:Junk_error_sum}
        e_{h}^{n}
        =
        (A_h)^n e_{h}^{0}
        -
        \sum_{k=0}^{n-1}(A_{h})^{n-k-1}\widetilde{S}^{-1}\hat{r}_{h}^{k}
        -
        \sum_{k=0}^{n-1}(A_{h})^{n-k-1}\widetilde{S}^{-1}\widehat{R}_{h}^{k}
    \end{equation}
    for $1 \leq n \leq N_h$. From~\eqref{eq:stability} and Proposition~\ref{prop:Junk_error}, it follows that
    \begin{align}
        \| e_{h}^{n} \|_{X_h}
        & \leq
        \| e_{h}^{0} \|_{X_h}
        +
        \sum_{k=0}^{n-1}\| \hat{r}_{h}^{k} \|_{X_h}
        +
        \sum_{k=0}^{n-1}\| \widehat{R}_{h}^{k} \|_{X_h} \\
        & \leq 
        \| f_{h}^{0}-\hat{f}_{h}(0) \|_{X_h}
        +
        \sum_{k=0}^{n-1}C_T h^{7/2} \\
        & \leq
        \| f_{h}^{0}-\hat{f}_{h}(0) \|_{X_h}
        +C_T h^{3/2}.
    \end{align}
    This proves~\eqref{eq:thm_Junk_f}, which combined with~\eqref{eq:macro-intro} and~\eqref{eq:Ph_moments} yields
    \begin{align}
        \| U_{h}^{n}-u(t_n)-h^2 w(t_n) \|_{\ell_h^2}
        & \leq C_T h^{1/2}, \\
        \| P_{h}^{n}-p(t_n)-h^2q(t_n) \|_{\ell_h^2}
        & \leq C_T h^{-1/2}
    \end{align}
    for $0 \leq n \leq N_h$. The bounds~\eqref{eq:thm_Junk_U} and~\eqref{eq:thm_Junk_P} follow immediately from these.
\end{proof}

Theorem~\ref{thm:Junk_convergence} implies that if we choose $f_{h}^{0}$ so that $\| f_{h}^{0}-\hat{f}_h(0) \|_{X_h}=O(h^{3/2})$, we have the $1/2$-order convergence for the velocity
\begin{equation}
    \| U_{h}^{n}-u(t_n) \|_{\ell_h^2} \leq C_T h^{1/2}.
\end{equation}
This convergence rate does not improve by choosing the initial data $f_{h}^{0}$ meticulously. The pressure $P_{h}^{n}$, on the other hand, is not guaranteed to converge to the pressure $p$ of the Stokes equations~\eqref{eq:stokes-intro}. However, as mentioned in the introduction, formal analyses and numerical experiments suggest $O(h^2)$ error for the velocity and $O(h)$ error for the pressure.

\subsection{Main theorem}

As noted above, the $1/2$-order convergence for the velocity guaranteed by Theorem~\ref{thm:Junk_convergence} is suboptimal. For the pressure, it does not even guarantee the convergence. In this paper, we improve Theorem~\ref{thm:Junk_convergence} and prove the optimal second-order convergence for the velocity and first-order convergence for the pressure.

The basic strategy of our proof is to refine the prediction function $\hat{f}_h$ by adding Stokes and Knudsen layer corrections. This refinement improves the $O(h^3)$ pointwise estimates in Proposition~\ref{prop:Junk_error} for the boundary residual $\widehat{R}_{h}^{n}$ to $O(h^5)$. Then by similar calculations as in the proof of Theorem~\ref{thm:Junk_convergence}, the second-order convergence for the velocity and first-order convergence for the pressure follow immediately. The construction of the Stokes correctors is inspired by the boundary error analysis technique of Dong, Ying, and Zhang for a MAC scheme for Stokes equations with Dirichlet boundary conditions~\cite{DongYingZhang2020}. The Knudsen layer correctors are discrete analogues of the Knudsen layer correctors in the kinetic theory of gases; see~\cite[Section~3]{Sone2002}.

To state our main theorem, let us first introduce a terminology on the initial data $f_{h}^{0}$ for the LBM scheme~\eqref{eq:LBM-compact}.

\begin{definition}\label{def:3rd-initialization}
    We call $f_{h}^{0} \in X_h$ a third-order initialization if there exists a constant $C_T>0$ depending only on $T>0$ such that
    \begin{equation}
        \| f_{h}^{0}-\hat{f}_{h}(0) \|_{X_h} \leq C_T h^3,
    \end{equation}
    where $\hat{f}_{h}$ is the prediction function defined by~\eqref{eq:Junk_prediction} and~\eqref{eq:Ph}.
\end{definition}

An example of a third-order initialization is given by
\begin{equation}
    f_{h}^{0}(z) \coloneqq \sum_{k=0}^{2}h^k f^{(k)}(0,z),
    \qquad
    z \in \Omega_h,
\end{equation}
where $f^{(k)}$ is defined by~\eqref{eq:Junk_prediction}.

We now state our main theorem, which gives the optimal second-order convergence for the velocity and first-order convergence for the pressure.

\begin{theorem}\label{thm:optimal_convergence}
    Assume that $\tau>1/2$ and that the solution $(u,p)$ to~\eqref{eq:stokes-intro} is $C^5$-regular in $t$ and $x$. Let $f_{h}^{0}$ be a third-order initialization and let $f_{h}^{n}$ be the solution to the LBM scheme~\eqref{eq:LBM-compact}. Then there exists a constant $C_T>0$ depending only on $T$ such that
    \begin{equation}
        \| U_{h}^{n}-u(t_n) \|_{\ell_{h}^{2}} \leq C_T h^2
    \end{equation}
    and
    \begin{equation}
        \| P_{h}^{n}-p(t_n) \|_{\ell_{h}^{2}} \leq C_T h
    \end{equation}
    for $0 \leq n \leq N_h$. Here, $t_n=nh^2$ and $(U_{h}^{n},P_{h}^{n})$ are defined by~\eqref{eq:macro-intro}.
\end{theorem}

The rest of the paper is organized as follows. In Section~\ref{sec:correctors}, we construct our refined prediction function, which is the core of our analysis. Then combining this with Proposition~\ref{prop:stability}, Theorem~\ref{thm:optimal_convergence} is proved in Section~\ref{sec:convergence} in a similar fashion as Theorem~\ref{thm:Junk_convergence}.

\section{Refined prediction function}\label{sec:correctors}

\subsection{Prediction operator and residual operators}
We first define the prediction operator $\mathcal{P}_h$ as follows. Let $(U,P)$ be a solution to the Stokes equations
\begin{equation}\label{eq:Stokes}
\begin{dcases}
    \partial_tU+\nabla P=\nu\Delta U &\text{in $(0,T)\times\Omega$},\\
    \nabla\cdot U=0 &\text{in $(0,T)\times\Omega$}\\
\end{dcases}
\end{equation}
with periodicity in $x$ and the normalization $\int_\Omega P(t,x,y)\dd x \dd y=0$. Note that we do not specify the boundary and the initial data. Let $(W,Q)$ be the solution to the modified Stokes equations
\begin{equation}\label{eq:higher_Stokes}
    \begin{dcases}
        \partial_t W+\nabla Q=\nu \Delta W-\frac{\tau}{3}\nabla P & \text{in $(0,T) \times \Omega$}, \\
        \nabla \cdot W=-\partial_t P & \text{in $(0,T) \times \Omega$}, \\
        W=0 & \text{on $(0,T) \times \partial \Omega$}, \\
        W(0,\cdot)=W_{\mathrm{in}} & \text{in $\Omega$}
    \end{dcases}
\end{equation}
with periodicity in $x$ and the normalization $\int_{\Omega}Q(t,x,y) \dd x \dd y=0$. The initial data $W_{\mathrm{in}}$ is required to satisfy the standard compatibility conditions of sufficiently high order and to depend linearly on $(U,P)$. An explicit construction of such initial data is rather lengthy and is omitted here. Now, define $\hat{f}^{(k)}[U,P] \colon [0,T] \times \Omega \to \R^9$ as $f^{(k)}$ in~\eqref{eq:Junk_prediction} with $(u,p)$ and $(w,q)$ replaced by $(U,P)$ and $(W,Q)$. Then set
\begin{equation}\label{def:P}
    \mathcal{P}_h[U,P]
    \coloneqq
    \sum_{k=0}^{5}h^k \hat{f}^{(k)}[U,P].
\end{equation}
We note that $\mathcal{P}_h[U,P]$ is affine in $(U,P)$, or more precisely, $\mathcal{P}_h[U,P] - F^{\mathrm{eq}}(1,0)$ is linear in $(U,P)$.

Next, we define the residual operators $\mathcal{E}_h$, $r_h$, $R_h$, and $R_{h}^{\pm}$. For either a function $F \colon [0,T] \times \Omega \to \R^9$ or $F \colon [0,T] \times \Omega_h \to \R^9$, we introduce the total residual of $F$ as the function
\begin{equation}
    \mathcal{E}_h[F]=(\mathcal{E}_{h,i}[F])_{i=0}^{8} \colon [0,T-h^2] \times \Omega_h \to \R^9
\end{equation}
defined by
\begin{equation}\label{def:E}
    \mathcal{E}_h[F](t)
    \coloneqq
    \widetilde{S} F(t+h^2)-\mathscr{C} F(t).
\end{equation}
In components, the definition can be written as
\begin{equation}
    \mathcal{E}_{h,i}[F](t,z)
    =
    \begin{dcases}
        F_i(t+h^2,z+c_i h)-(\mathscr{C}F(t,z))_i & (z+c_i h \in \Omega_h), \\
        F_{i^*}(t+h^2,z)-(\mathscr{C}F(t,z))_i & (z+c_i h \notin \Omega_h)
    \end{dcases}
\end{equation}
The total residual $\mathcal{E}_h[F]$ can be decomposed as
\begin{equation}\label{eq:E_rR}
    \mathcal{E}_h[F] = r_h[F] + R_h[F],
\end{equation}
where
\begin{align}
    r_h[F]
    =(r_{h,i}[F])_{i=0}^{8} & \colon [0,T-h^2] \times \Omega_h \to \R^9, \\
    R_h[F]
    =(R_{h,i}[F])_{i=0}^{8} & \colon [0,T-h^2] \times \Omega_h \to \R^9
\end{align}
are defined by
\begin{align}
    r_{h,i}[F](t,z)
    & \coloneqq
    \begin{dcases}
        \mathcal{E}_{h,i}[F](t,z) & (z+c_i h \in \Omega_h), \\
        0 & (z+c_i h \notin \Omega_h),
    \end{dcases}
    \label{def:interior-residual}
    \\
    R_{h,i}[F](t,z)
    & \coloneqq
    \begin{dcases}
        0 & (z+c_i h \in \Omega_h), \\
        \mathcal{E}_{h,i}[F](t,z) & (z+c_i h \notin \Omega_h).
    \end{dcases}
    \label{def:boundary-residual}
\end{align}
The functions $r_h[F]$ and $R_h[F]$ are called the interior residual and the boundary residual, respectively. We note that these are related to the functions $\hat{r}_{h}^{n},\widehat{R}_{h}^{n} \in X_h$ appearing in Proposition~\ref{prop:Junk_error} as follows:
\begin{equation}
    \hat{r}_{h}^{n}(z)
    =
    r_h[\hat{f}_h](t_n,z),
    \qquad
    \widehat{R}_{h}^{n}(z)
    =
    R_h[\hat{f}_h](t_n,z).
\end{equation}
For $x\in\T_h$, we write
\begin{equation}
    z_x^-
    \coloneqq
    \left(x,\frac h2\right),
    \qquad
    z_x^+
    \coloneqq
    \left(x,1-\frac h2\right).
\end{equation}
We also introduce
\begin{align}
    R_h^{-}[F]=(R_{h,i}^{-}[F])_{i \in \{ N,NE,NW \}}
    & \colon [0,T-h^2] \times \T_h \to \R^3, \\
    R_h^{+}[F]=(R_{h,i}^{+}[F])_{i \in \{ S,SW,SE \}}
    & \colon [0,T-h^2] \times \T_h \to \R^3
\end{align}
defined by
\begin{align}\label{eq:full_to_partial_residual}
\begin{aligned}
    R_{h,i}^{-}[F](t,x)
    & \coloneqq
    R_{h,i^*}[F](t,z_{x}^{-}), \qquad i \in \{ N,NE,NW \}, \\
    R_{h,i}^{+}[F](t,x)
    & \coloneqq
    R_{h,i^*}[F](t,z_{x}^{+}), \qquad i \in \{ S,SW,SE \}.
\end{aligned}
\end{align}
We remind the reader that $i^*$ is the index such that $c_{i^*}=-c_i$.

We next state two lemmas on asymptotic expansions of the interior and boundary residuals of $\mathcal{P}_h[U,P]$ defined above. To this end, we introduce two operators $\mathcal{C}_{\pm}$. For $G=(G_x,G_y)^T \in \R^2$, let
\begin{equation}\label{eq:def-Cm}
    \mathcal C_-G
    \coloneqq
    \begin{pmatrix}
        F_{N}^{\mathrm{eq}}(0,G) - F_{S}^{\mathrm{eq}}(0,G) \\
        F_{NE}^{\mathrm{eq}}(0,G) - F_{SW}^{\mathrm{eq}}(0,G) \\
        F_{NW}^{\mathrm{eq}}(0,G) - F_{SE}^{\mathrm{eq}}(0,G)
    \end{pmatrix}
    =
    \begin{pmatrix}
        \dfrac23G_y\\[1mm]
        \dfrac16(G_x+G_y)\\[1mm]
        \dfrac16(-G_x+G_y)
    \end{pmatrix}
\end{equation}
and
\begin{equation}\label{eq:def-Cp}
    \mathcal C_+G
    \coloneqq
     \begin{pmatrix}
        F_{S}^{\mathrm{eq}}(0,G) - F_{N}^{\mathrm{eq}}(0,G) \\
        F_{SW}^{\mathrm{eq}}(0,G) - F_{NE}^{\mathrm{eq}}(0,G) \\
        F_{SE}^{\mathrm{eq}}(0,G) - F_{NW}^{\mathrm{eq}}(0,G)
    \end{pmatrix}
    =
    \begin{pmatrix}
        -\dfrac23G_y\\[1mm]
        -\dfrac16(G_x+G_y)\\[1mm]
        \dfrac16(G_x-G_y)
    \end{pmatrix}.
\end{equation}
Let $(U,P)$ and $(W,Q)$ be the functions defined by~\eqref{eq:Stokes} and~\eqref{eq:higher_Stokes}, respectively, and write
\begin{equation}
    r_h[U,P]
    \coloneqq
    r_h[\mathcal{P}_h[U,P]],
    \qquad
    R_h^{\pm}[U,P]
    \coloneqq
    R_h^{\pm}[\mathcal{P}_h[U,P]].
\end{equation}
Denote by
\begin{equation}
    G^{\pm}
    \colon
    [0,T] \times \T \to \R^2
\end{equation}
the boundary data of $U$, that is,
\begin{equation}
    G^{-}(t,x)
    \coloneqq
    U(t,x,0),
    \qquad
    G^{+}(t,x)
    \coloneqq
    U(t,x,1).
\end{equation}

\begin{lemma}[Homogeneous boundary data]\label{lem:pred_error}
    Assume that $G^{\pm} \equiv 0$. For the interior residual $r_h[U,P]$, we have
    \begin{equation}
        r_{h}[U,P]=O(h^6)
    \end{equation}
    uniformly in $[0,T-h^2] \times \Omega_h$. Moreover, for the boundary residuals $R_{h}^{\pm}[U,P]$, there exist smooth functions
    \begin{equation}
        R^{\pm,(3)}[U,P]
        \colon
        [0,T] \times \T \to \R^3
    \end{equation}
    independent of $h>0$ such that
    \begin{equation}
        R_{h}^\pm[U,P]
        =
        h^3 R^{\pm,(3)}[U,P]
        +
        O(h^4)
    \end{equation}
    uniformly in $[0,T-h^2] \times \T_h$.
\end{lemma}

\begin{proof}
    The proof is basically an application of Taylor's theorem to $\mathcal{P}_h[U,P]$. For details, we refer to the proof of Proposition~\ref{prop:Junk_error} in~\cite{JunkYang2008}, which further relies on the asymptotic analysis in~\cite{JunkYang2005}.
\end{proof}

\begin{lemma}[General boundary data]\label{lem:pred_error_inhom}
    For the interior residual $r_h[U,P]$, we have
    \begin{equation}
        r_{h}[U,P]=O(h^6)
    \end{equation}
    uniformly in $[0,T-h^2] \times \Omega_h$. For the boundary residuals $R_{h}^{\pm}[U,P]$, we have
    \begin{equation}\label{eq:residual_Cpm}
        R_{h}^\pm[U,P]
        =
        h\mathcal{C}_{\pm} G^{\pm}
        +
        O(h^2)
    \end{equation}
    uniformly in $[0,T-h^2] \times \T_h$.
\end{lemma}

\begin{proof}
    The estimate of the interior residual $r_{h}[U,P]$ can be done similarly to the case of $G^{\pm} \equiv 0$ in Lemma~\ref{lem:pred_error}. Here, let us only consider the boundary residuals. By the definition of $\mathcal{P}_h[U,P]$, we have
    \begin{equation}
        \mathcal{P}_h[U,P]
        =
        F^{\mathrm{eq}}(1,0)
        +
        h F^{\mathrm{eq}}(0,U)
        +
        O(h^2).
    \end{equation}
    Therefore, for $0 \leq t \leq T-h^2$, $i \in \{ 0,\ldots,8 \}$, and $z \in \Omega_h$ with $z+c_i h \notin \Omega_h$, we have
    \begin{align}
        R_{h,i}[\mathcal{P}_h[U,P]](t,z)
        & =
        \mathcal{P}_{h,i^*}[U,P](t+h^2,z)
        -
        (\mathscr C \mathcal{P}_h[U,P](t,z))_i \\
        & =
        F_{i^*}^{\mathrm{eq}}(1,0)
        -
        (\mathscr C F^{\mathrm{eq}}(1,0))_i
        \\
        & \qquad
        +h \left\{ F_{i^*}^{\mathrm{eq}}(0,U(t,z))
        -
        (\mathscr C F^{\mathrm{eq}}(0,U(t,z)))_i \right\}
        +
        O(h^2) \\
        & =
        h\left\{ F_{i^*}^{\mathrm{eq}}(0,U(t,z)) - F_{i}^{\mathrm{eq}}(0,U(t,z)) \right\}
        +
        O(h^2).
    \end{align}
    Here, we used $\mathscr{C}F^{\mathrm{eq}}(\rho_0,j_0)=F^{\mathrm{eq}}(\rho_0,j_0)$ for any $\rho_0 \in \R$ and $j_0 \in \R^2$. Since
    \begin{equation}
        U(t,z_{x}^{\pm})
        =
        G^{\pm}(t,x)+O(h)
    \end{equation}
    uniformly in $(t,x) \in [0,T] \times \T_h$ by Taylor's theorem, taking into account~\eqref{eq:full_to_partial_residual},~\eqref{eq:def-Cm}, and~\eqref{eq:def-Cp}, we obtain~\eqref{eq:residual_Cpm}.
\end{proof}

In particular, applying Lemma~\ref{lem:pred_error} to the prediction function $\hat{f}_h$ defined by~\eqref{eq:Junk_prediction} and~\eqref{eq:Ph}, we immediately obtain the following proposition.

\begin{proposition}\label{prop:leading-boundary}
    Let
    \begin{equation}
        \hat{r}_{h} \coloneqq r_h [ \hat{f}_h ],
        \qquad
        \widehat{R}_{h}^{\pm}
        \coloneqq
        R_{h}^{\pm}[\hat{f}_h].
    \end{equation}
    For the interior residual $\hat{r}_h$, we have
    \begin{equation}\label{eq:boundary-leading}
        \hat{r}_{h}=O(h^6)
    \end{equation}
    uniformly in $[0,T-h^2] \times \Omega_h$. Moreover, for the boundary residuals $\widehat R_{h}^{\pm}$, there exist smooth functions
    \begin{equation}
        R^{\pm,(3)}
        \colon
        [0,T] \times \T \to \R^3
    \end{equation}
    independent of $h>0$ such that
    \begin{equation}\label{eq:boundary-expansion-pm}
        \widehat R_{h}^\pm
        =
        h^3 R^{\pm,(3)}
        +
        O(h^4)
    \end{equation}
    uniformly in $[0,T-h^2] \times \T_h$.
\end{proposition}

\subsection{First Stokes corrector}\label{sec:first-stokes}

In this section, we construct our first corrector to the prediction function $\hat{f}_h$ as a solution to the Stokes equations with appropriately chosen Dirichlet boundary data. The role of the corrector is to cancel the macroscopic part of the term $h^3 R^{\pm,(3)}$ in~\eqref{eq:boundary-expansion-pm}.

Let
\begin{equation}\label{eq:def-kappa-minus-plus}
    \kappa_-
    \coloneqq
    \begin{pmatrix}
        2\\-1\\-1
    \end{pmatrix},
    \qquad
    \kappa_+
    \coloneqq
    \begin{pmatrix}
        2\\-1\\-1
    \end{pmatrix}.
\end{equation}
Their embeddings into $\R^9$ are denoted by adding tildes:
\begin{equation}\label{eq:kappa-embeddings}
    \tilde \kappa_-
    \coloneqq
    2e_N-e_{NE}-e_{NW},
    \qquad
    \tilde \kappa_+
    \coloneqq
    2e_S-e_{SW}-e_{SE},
\end{equation}
where $(e_i)_{i=0}^{8}$ is the standard basis of $\R^9$. They have zero density and zero momentum:
\begin{equation}\label{eq:kappa-zero-moments}
    \sum_{i=0}^8\tilde \kappa_{\pm,i}=0,
    \qquad
    \sum_{i=0}^8c_i \tilde \kappa_{\pm,i}=0.
\end{equation}
We note that the maps $\mathcal C_\pm$ defined by~\eqref{eq:def-Cm} and~\eqref{eq:def-Cp} are injective and that $\lambda \coloneqq (1,-2,-2)^T \in (\operatorname{Ran} \mathcal{C}_{\pm})^{\perp}$. Since $\lambda \cdot \kappa_{\pm} = 6 \neq 0$, we have the direct sum decomposition
\begin{equation}\label{eq:boundary-direct-sum}
    \R^3
    =
    \operatorname{Ran}\mathcal C_\pm
    \oplus
    \operatorname{span}\{\kappa_\pm\}.
\end{equation}
Hence there exist unique functions
\[
    g_0^\pm \colon [0,T] \times \T \to \R^2,
    \qquad
    \alpha_\pm \colon [0,T] \times \T \to \R
\]
such that
\begin{equation}\label{eq:def-g0-alpha-pm}
    R^{\pm,(3)}+\mathcal C_\pm g_0^\pm
    =
    \alpha_\pm\kappa_\pm,
\end{equation}
where $R^{\pm,(3)}$ are the functions in Proposition~\ref{prop:leading-boundary}. We prove that $g_{0}^{\pm}$ satisfy the admissibility condition for the Stokes equations~\eqref{eq:first-Stokes-corrector}, namely, that the total outward flux vanishes.

\begin{lemma}\label{lem:first-flux-compatibility}
    The functions $g_0^\pm$ satisfy
    \begin{equation}\label{eq:first-flux-compatibility}
        -\int_{\T}
            (g_0^-)_y(t,x) \dd x
        +
        \int_{\T}
            (g_0^+)_y(t,x) \dd x
        =
        0
    \end{equation}
    for every $t\in[0,T]$.
\end{lemma}

\begin{proof}
    Let
    \begin{equation}\label{eq:mu}
        \mu(R)
        \coloneqq
        R_a+R_b+R_c, \qquad R=
        \begin{pmatrix}
            R_a \\
            R_b \\
            R_c
        \end{pmatrix}
        \in \R^3.
    \end{equation}
    Then since
    \[
        \mu(\mathcal C_{\pm} g_{0}^{\pm})=\mp (g_{0}^{\pm})_y, \qquad
        \mu(\kappa_\pm)=0,
    \]
    equation~\eqref{eq:def-g0-alpha-pm} implies
    \begin{equation}\label{eq:g0-normal-flux}
        -(g_0^-)_y
        =
        \mu(R^{-,(3)}),
        \qquad
        (g_0^+)_y
        =
        \mu(R^{+,(3)}).
    \end{equation}
    For a lattice function $f \colon \Omega_h \to \R^9$, define its total mass by
    \begin{equation}\label{eq:operator_mass}
        M_h[f]
        \coloneqq
        h^2
        \sum_{z\in\Omega_h}
        \sum_{i=0}^8 f_i(z).
    \end{equation}
    It is easy to check that the streaming and collision processes preserve the total mass, i.e.,
    \begin{equation}
        M_h[\widetilde Sf]
        =
        M_h[f],
        \qquad
        M_h[\mathscr Cf]
        =
        M_h[f].
        \label{eq:mass-preservation-operators}
    \end{equation}
    Let $\hat{f}_h$ be the prediction function defined by~\eqref{eq:Ph}. We recall from~\eqref{eq:Ph_moments} that
    \begin{equation}\label{eq:prediction-density}
        \sum_{i=0}^8\hat f_{h,i}
        =
        1+3h^2p+3h^4q.
    \end{equation}
    By the periodicity in $x$ and the half-way configuration of our lattice sites, $h^2 \sum_{z \in \Omega_h}\varphi(z)$ is a second-order approximation of $\int_{\Omega}\varphi(x,y) \dd x \dd y$ for a function $\varphi \colon \Omega \to \R$. Since $p$ and $q$, and hence $\partial_t p$ and $\partial_t q$, have zero spatial mean, this implies
    \begin{equation}
        h^2
        \sum_{z\in\Omega_h}
        \partial_t p(t,z)
        =
        O(h^2),
        \qquad
        h^2
        \sum_{z\in\Omega_h}
        \partial_t q(t,z)
        =
        O(h^2).
    \end{equation}
    Consequently, by Taylor's theorem, we obtain
    \begin{align}\label{eq:fhat_time_taylor}
        \begin{aligned}
            M_h\left[ \hat f_h(t+h^2) \right]
            -
            M_h\left[ \hat f_h(t) \right]
            & =
            3h^2
            \left\{
                h^2 \sum_{z \in \Omega_h}
                \left(
                    p(t+h^2,z) - p(t,z)
                \right)
            \right\}
            \\
            & \qquad +
            3h^4
            \left\{
                h^2 \sum_{z \in \Omega_h}
                \left(
                    q(t+h^2,z) - q(t,z)
                \right)
            \right\}
            \\
            & =
            O(h^6).
        \end{aligned}
    \end{align}
    By definition, we have
    \begin{equation}
        \mathcal{E}_h[\hat{f}_h](t) = r_h[\hat{f}_h](t) + R_h[\hat{f}_h](t).
    \end{equation}
    Applying $M_h$ to this and using~\eqref{def:E} and~\eqref{eq:mass-preservation-operators}, we obtain
    \begin{equation}
        M_h[\hat{f}_h(t+h^2)] - M_h[\hat{f}_h(t)] = M_h[r_h[\hat{f}_h](t)] + M_h[R_h[\hat{f}_h](t)].
    \end{equation}
    Hence by~\eqref{eq:full_to_partial_residual},~\eqref{eq:boundary-leading},~\eqref{eq:fhat_time_taylor}, and the fact that $R_h[\hat{f}_h(t)]$ is supported on the set of $z \in \Omega_h$ with $z+c_i h \notin \Omega_h$, we get
    \begin{equation}
        h^2
        \sum_{j=0}^{M-1}
        \left\{
            \mu\left(\widehat R_{h}^-(t,jh)\right)
            +
            \mu\left(\widehat R_{h}^+(t,jh)\right)
        \right\}
        =
        O(h^6).
        \label{eq:boundary-mass-balance}
    \end{equation}
    Using~\eqref{eq:boundary-expansion-pm}, this becomes
    \[
        h^4
        \left[
            h\sum_{j=0}^{M-1}
            \left\{
                \mu(R^{-,(3)}(t,jh))
                +
                \mu(R^{+,(3)}(t,jh))
            \right\}
        \right]
        =
        O(h^5).
    \]
    Letting $h\to0$ gives
    \begin{equation}
        \int_{\T}
            \mu(R^{-,(3)}(t,x))\dd x
        +
        \int_{\T}
            \mu(R^{+,(3)}(t,x))\dd x
        =
        0.
    \end{equation}
    The conclusion follows from~\eqref{eq:g0-normal-flux}.
\end{proof}

Let $v_{0,\mathrm{in}}$ be a smooth divergence-free vector that is periodic in $x$. Taking $g_{0}^{\pm}$ as the boundary data, let $(v_0,\pi_0)$ be the solution to the Stokes equations
\begin{equation}\label{eq:first-Stokes-corrector}
\begin{dcases}
    \partial_t v_0+\nabla\pi_0
    =
    \nu\Delta v_0
    &\text{in $(0,T)\times\Omega$},
    \\
    \nabla\cdot v_0=0
    &\text{in $(0,T)\times\Omega$},
    \\
    v_0(t,x,0)=g_0^-(t,x)
    &\text{for $(t,x)\in(0,T)\times\T$},
    \\
    v_0(t,x,1)=g_0^+(t,x)
    &\text{for $(t,x)\in(0,T)\times\T$},
    \\
    v_0(0,\cdot)=v_{0,\mathrm{in}}
    &\text{in $\Omega$}
\end{dcases}
\end{equation}
with periodicity in $x$ and normalization $\int_\Omega\pi_0(t,x,y)\dd x\dd y=0$. The necessary admissibility condition is guaranteed by Lemma~\ref{lem:first-flux-compatibility}. The initial data $v_{0,\mathrm{in}}$ is required to satisfy the standard compatibility conditions of sufficiently high order. Set
\begin{equation}\label{eq:first-outer-correction}
    u^{[1]}
    \coloneqq
    u+h^2v_0,
    \qquad
    p^{[1]}
    \coloneqq
    p+h^2\pi_0.
\end{equation}
We refine the prediction function $\hat{f}_h$ to
\begin{equation}\label{eq:corrected_pred_1}
    \hat f_h^{[1]}
    \coloneqq
    \mathcal P_h[u^{[1]},p^{[1]}]
    =
    \mathcal P_h[u+h^2 v_0,p+h^2 \pi_0].
\end{equation}
Concerning the interior and the boundary residuals of $\hat{f}_{h}^{[1]}$, we have the following proposition.

\begin{proposition}\label{prop:asymptotics-[1]}
    Let
    \begin{equation}
        \hat{r}_{h}^{[1]} \coloneqq r_h [ \hat{f}_h^{[1]} ],
        \qquad
        \widehat{R}_{h}^{\pm,[1]}
        \coloneqq
        R_{h}^{\pm}[\hat{f}_h^{[1]}].
    \end{equation}
    For the interior residual $\hat{r}_h^{[1]}$, we have
    \begin{equation}\label{eq:interior-after-first-Stokes-kappa}
        \hat{r}_{h}^{[1]}=O(h^6)
    \end{equation}
    uniformly in $[0,T-h^2] \times \Omega_h$. Moreover, for the boundary residuals $\widehat R_{h}^{\pm,[1]}$, we have
    \begin{equation}\label{eq:boundary-after-first-Stokes-kappa}
        \widehat R_{h}^{\pm,[1]}
        =
        h^3 \alpha_{\pm} \kappa_{\pm}
        +
        O(h^4)
    \end{equation}
    uniformly in $[0,T-h^2] \times \T_h$, where $\alpha_{\pm}$ are defined by~\eqref{eq:def-g0-alpha-pm}.
\end{proposition}

\begin{proof}
    We first note that
    \begin{equation}\label{eq:fh1-fh}
        \hat{f}_{h}^{[1]}
        -
        \hat{f}_{h}
        =
        h^2 \left( \mathcal{P}_h[v_0,\pi_0] - F^{\mathrm{eq}}(1,0) \right).
    \end{equation}
    Since the constant equilibrium $F^{\mathrm{eq}}(1,0)$ has zero interior and boundary residuals, the linearity of the residual operators $r_h$ and $R_{h}^{\pm}$ gives
    \begin{equation}
        r_{h}[\hat{f}_{h}^{[1]}]
        -
        r_{h}[\hat{f}_h]
        =
        h^2 r_{h}[\mathcal{P}_h[v_0,\pi_0]]
    \end{equation}
    and
    \begin{equation}
        R_{h}^{\pm}[\hat{f}_{h}^{[1]}]
        -
        R_{h}^{\pm}[\hat{f}_h]
        =
        h^2 R_{h}^{\pm}[\mathcal{P}_h[v_0,\pi_0]].
    \end{equation}
    Then the interior estimate~\eqref{eq:interior-after-first-Stokes-kappa} follows immediately from Lemma~\ref{lem:pred_error_inhom} and~\eqref{eq:boundary-leading}. Similarly, the boundary estimate~\eqref{eq:boundary-after-first-Stokes-kappa} follows from Lemma~\ref{lem:pred_error_inhom},~\eqref{eq:boundary-expansion-pm}, and~\eqref{eq:def-g0-alpha-pm}.
\end{proof}

\subsection{First Knudsen layer correctors}\label{sec:phi0}

In this section, we construct Knudsen layer correctors. The role of the correctors is to cancel the $O(h^3)$ kinetic part $h^3 \alpha_{\pm} \kappa_{\pm}$ of the boundary residuals $\widehat{R}_{h}^{\pm,[1]}$ in~\eqref{eq:boundary-after-first-Stokes-kappa} while keeping the interior residual $\hat{r}_{h}^{[1]}$ at least $O(h^5)$. For this purpose, we need to introduce several discrete Knudsen layer profiles.

We remind the reader that we assume $\tau>1/2$ and $\gamma \in \R$ is defined by~\eqref{eq:def-gamma}. We set
\begin{equation}
    \Gamma
    \coloneqq
    |\gamma|, \qquad 0\leq\Gamma<1.
\end{equation}
We note that~\eqref{eq:kappa-zero-moments} implies $\rho_{\tilde \kappa_{\pm}}=0$ and $j_{\tilde \kappa_{\pm}}=0$. Hence
\begin{equation}\label{eq:collision-kappa}
    \mathscr C\tilde \kappa_\pm
    =
    \gamma \tilde \kappa_\pm.
\end{equation}
For $\ell \in \{ 0,\ldots,M-1 \}$, define the normal lattice distances from the two walls by
\begin{equation}
    m_-(\ell)\coloneqq\ell,
    \qquad
    m_+(\ell)\coloneqq M-1-\ell.
\end{equation}

\subsubsection{Principal Knudsen layer correctors}
\label{sec:first-non-tangential-Knudsen}
We define the principal Knudsen layer profiles by
\begin{equation}\label{eq:def-Phi0-pm}
    \Phi_0^\pm(t,x,m)
    \coloneqq
    -\alpha_\pm(t,x)\gamma^m \tilde \kappa_\pm,
    \qquad
    (t,x,m) \in [0,T] \times \T \times \mathbb N_0
\end{equation}
and their lattice realizations by
\begin{equation}\label{eq:def-Phi0h-pm}
    \Phi_{0,h}^\pm(t,z_{j,\ell})
    \coloneqq
    \Phi_0^\pm
    \left(
        t,jh,m_\pm(\ell)
    \right),
    \qquad
    (t,j,\ell) \in [0,T] \times \mathbb{Z}/M\mathbb{Z} \times \{ 0,\ldots,M-1 \}.
\end{equation}
Here and in what follows, we adopt the convention $0^0=1$. We note that $\Phi_{0,h}^{\pm}$ decay exponentially: we have
\begin{equation}\label{eq:Phi0-decay}
    |\Phi_{0,h}^\pm(t,z_{j,\ell})|
    \leq
    C_T\Gamma^{m_\pm(\ell)}
\end{equation}
for some constant $C_T>0$ depending only on $T$. Thus these corrections are localized within $O(1)$ lattice layers of the boundary. Concerning the interior and the boundary residuals of the principal Knudsen layer correctors $\Phi_{0,h}^{-}$, we have the following two propositions.

\begin{proposition}
    \label{prop:boundary-Phi0}
    We have
    \begin{equation}\label{eq:boundary-Phi0}
        R_{h}^{\pm}[\Phi_{0,h}^{\pm}]
        =
        -\alpha_{\pm} \kappa_{\pm}
        +
        O(h^2)
    \end{equation}
    uniformly in $[0,T-h^2] \times \T_h$.
\end{proposition}

\begin{proof}
    For $0 \leq t \leq T-h^2$, $i \in \{ 0,\ldots,8 \}$, and $z_{j,\ell} \in \Omega_h$ with $z_{j,\ell}+c_i h \notin \Omega_h$, taking into account~\eqref{eq:kappa-embeddings} and~\eqref{eq:collision-kappa}, we obtain
    \begin{align}
        R_{h,i}[\Phi_{0,h}^{\pm}](t,z_{j,\ell})
        & =
        \Phi_{0,h,i^*}^{\pm}(t+h^2,z_{j,\ell})
        -
        (\mathscr C \Phi_{0,h}^{\pm}(t,z_{j,\ell}))_i \\
        & =
        -\alpha_{\pm}(t+h^2,jh) \tilde{\kappa}_{\pm,i^*}
        +
        \gamma \alpha_{\pm}(t,jh) \tilde{\kappa}_{\pm,i} \\
        & =
        -\alpha_{\pm}(t+h^2,jh) \tilde{\kappa}_{\pm,i^*}.
    \end{align}
    Noting the definition~\eqref{eq:full_to_partial_residual}, we obtain~\eqref{eq:boundary-Phi0} by applying Taylor's theorem.
\end{proof}

\begin{proposition}
\label{prop:interior-Phi0}
For $i \in \{ 0,\ldots,8 \}$ and $z_{j,\ell} \in \Omega_h$ with $z_{j,\ell}+c_i h \in \Omega_h$, we have
\begin{align}\label{eq:interior-Phi0-exact}
    \begin{aligned}
        r_{h,i}[\Phi_{0,h}^{\pm}](t,z_{j,\ell})
        &=
        -
        \left\{
            \alpha_\pm
            (
            t+h^2,jh+c_{ix}h
            )
            -
            \alpha_\pm(t,jh)
        \right\}
        \gamma^{m_\pm(\ell)+1} \tilde\kappa_{\pm,i}
        \\
        &=
        -h
        \partial_x \alpha_\pm(t,jh)
        \gamma^{m_\pm(\ell)+1}
        \left(
            V_x\tilde\kappa_\pm
        \right)_i
        +
        O\left(
            h^2\Gamma^{m_\pm(\ell)+1}
        \right)
    \end{aligned}
\end{align}
uniformly for $0\leq t\leq T-h^2$. Here
\begin{equation}\label{eq:def-Vx}
    V_x
    \coloneqq
    \operatorname{diag}
    \left(
        c_{0x},\ldots,c_{8x}
    \right).
\end{equation}
\end{proposition}

\begin{proof}
We prove the assertion for $\Phi_{0,h}^{-}$; the calculations for $\Phi_{0,h}^{+}$ are similar. Let $z_{j,\ell} \in \Omega_h$ satisfy $z_{j,\ell}+c_i h\in\Omega_h$. By~\eqref{eq:collision-kappa},~\eqref{eq:def-Phi0-pm}, and~\eqref{eq:def-Phi0h-pm}, we obtain
\[
    \left(
        \mathscr C\Phi_{0,h}^{-}(t,z_{j,\ell})
    \right)_i
    =
    -\alpha_-(t,jh)
    \gamma^{m_-(\ell)+1}
    \tilde\kappa_{-,i}.
\]
Since~\eqref{eq:interior-Phi0-exact} is trivial for $i \in \{ 0,\ldots,8 \}$ with $\tilde\kappa_{-,i} = 0$, we assume $\tilde\kappa_{-,i} \neq 0$. In this case, we have $c_{iy}=1$ by~\eqref{eq:kappa-embeddings}, which implies
\[
    m_-(\ell+c_{iy})
    =
    m_-(\ell)+1.
\]
Therefore,
\[
    \Phi_{0,h,i}^{-}
    \left(
        t+h^2,
        z_{j,\ell}+c_i h
    \right)
    =
    -\alpha_-
    \left(
        t+h^2,
        jh+c_{ix} h
    \right)
    \gamma^{m_-(\ell)+1}
    \tilde\kappa_{-,i}.
\]
Thus, by definition, we have
\[
\begin{aligned}
    r_{h,i}[\Phi_{0,h}^{-}](t,z_{j,\ell})
    &=
    -
    \left\{
        \alpha_-
        \left(
            t+h^2,
            jh+c_{ix} h
        \right)
        -
        \alpha_-(t,jh)
    \right\}
    \gamma^{m_-(\ell)+1} \tilde\kappa_{-,i},
\end{aligned}
\]
which proves the first equality in~\eqref{eq:interior-Phi0-exact}. The second is an immediate consequence of Taylor's theorem.
\end{proof}

Using the principal Knudsen layer correctors $\Phi_{0,h}^{\pm}$ defined above, we refine the prediction function $\hat f_{h}^{[1]}$ further to
\begin{align}\label{eq:def-f11}
    \begin{aligned}
        \hat f_h^{[1,1]}
        & \coloneqq
        \hat f_h^{[1]}
        +
        h^3
        \left(
            \Phi_{0,h}^-+\Phi_{0,h}^+
        \right)
        \\
        & =
        \mathcal P_h[u+h^2 v_0,p+h^2 \pi_0]
        +
        h^3
        \left(
            \Phi_{0,h}^-+\Phi_{0,h}^+
        \right).
    \end{aligned}
\end{align}
Concerning the interior and the boundary residuals of $\hat{f}_{h}^{[1,1]}$, we have the following proposition.

\begin{proposition}\label{prop:asymptotics-[1,1]}
    Let
    \begin{equation}
        \hat{r}_{h}^{[1,1]} \coloneqq r_h [ \hat{f}_h^{[1,1]} ],
        \qquad
        \widehat{R}_{h}^{\pm,[1,1]}
        \coloneqq
        R_{h}^{\pm}[\hat{f}_h^{[1,1]}].
    \end{equation}
    For the interior residual $\hat{r}_h^{[1,1]}$, we have
    \begin{align}\label{eq:interior-[1,1]}
        \begin{aligned}
            \hat{r}_{h}^{[1,1]}(t,z_{j,\ell})
            & =
            \sum_{\sigma = +,-}
            \left\{
            -h^4 \partial_x \alpha_{\sigma}(t,jh)
            \gamma^{m_\sigma(\ell)+1}
                V_x\tilde\kappa_{\sigma}
            +
            O\left( h^5 \Gamma^{m_{\sigma}(\ell)+1} \right)
            \right\}
            \\
            & \qquad
            +
            O(h^6)
        \end{aligned}
    \end{align}
    uniformly in $(t,z_{j,\ell}) \in [0,T-h^2] \times \Omega_h$. Moreover, for the boundary residuals $\widehat R_{h}^{\pm,[1,1]}$, we have
    \begin{equation}\label{eq:boundary-[1,1]}
        \widehat R_{h}^{\pm,[1,1]}
        =
        O(h^4)
    \end{equation}
    uniformly in $[0,T-h^2] \times \T_h$.
\end{proposition}

\begin{proof}
    The assertions are immediately obtained by applying Propositions~\ref{prop:asymptotics-[1]}--\ref{prop:interior-Phi0} to~\eqref{eq:def-f11}. We remark that the contribution of $h^3 \Phi_{0,h}^{\mp}$ to $\widehat{R}_{h}^{\pm,[1,1]}$ is $O(h^3 \Gamma^{M-1})$ and is exponentially small (note that $M=h^{-1}$).
\end{proof}

\subsubsection{Tangential Knudsen layer correctors}\label{sec:phi1tan}

We next define the tangential Knudsen layer profiles. Their role is to eliminate the $O(h^4)$ terms $-h^4 \partial_x \alpha_\pm(t,jh) \gamma^{m_\pm(\ell)+1} V_x\tilde\kappa_\pm$ in the interior residual estimate~\eqref{eq:interior-[1,1]}.

For a function
\[
\Phi=\Phi(t,x,m) \colon [0,T] \times \T \times \mathbb{N}_0 \to \R^9,
\]
set
\begin{equation}
    (\mathscr T_h^{\pm}\Phi)_i(t,x,m)
    \coloneqq
    \Phi_i
    \left(
        t,x+c_{ix}h,m \mp c_{iy}
    \right)
\end{equation}
whenever the right-hand side is defined, i.e., when $m \mp c_{iy} \geq 0$. Let
\begin{equation}
    \mathscr L_h^{\pm}
    \coloneqq
    \mathscr T_h^{\pm} - \mathscr C.
\end{equation}
Formally applying Taylor's theorem in the tangential variable $x$ to the operator $\mathscr L_{h}^{\pm}$, we obtain
\begin{equation}\label{eq:Lh-expansion-minus}
    \mathscr L_h^{\pm}
    =
    \mathscr L_{(0)}^{\pm}
    +
    h\mathscr L_{(1)}^{\pm}
    +
    O(h^2),
\end{equation}
where
\begin{align}
    (\mathscr L_{(0)}^{\pm}\Phi)_i(t,x,m)
    &=
    \Phi_i(t,x,m \mp c_{iy})
    -
    (\mathscr C\Phi(t,x,m))_i,
    \label{eq:L0-minus}
    \\
    (\mathscr L_{(1)}^{\pm}\Phi)_i(t,x,m)
    &=
    c_{ix}\partial_x
    \Phi_i(t,x,m \mp c_{iy}).
    \label{eq:L1-minus}
\end{align}
In this way, we obtain two operators $\mathscr{L}_{{(0)}}^{\pm}$ and $\mathscr{L}_{{(1)}}^{\pm}$. We note that the expansion~\eqref{eq:Lh-expansion-minus} is only formal and will not be used in the actual proof.

Now, by~\eqref{eq:collision-kappa} and the fact that $c_{iy} = - 1$ if $\tilde{\kappa}_{+,i} \neq 0$ and $c_{iy} = 1$ if $\tilde{\kappa}_{-,i} \neq 0$, we observe that $\Phi_{0}^{\pm}$ defined by~\eqref{eq:def-Phi0-pm} solves
\begin{equation}
    \mathscr L_{(0)}^{\pm}\Phi_0^{\pm}=0.
\end{equation}
Furthermore, we have
\begin{equation}\label{eq:L1-Phi0-minus}
    \mathscr L_{(1)}^{\pm}\Phi_0^{\pm}
    =
    -
    \partial_x\alpha_{\pm}(t,x)
    \gamma^{m+1}
    V_x \tilde \kappa_{\pm}.
\end{equation}
This is actually how $-h^4 \partial_x \alpha_\pm(t,jh) \gamma^{m_\pm(\ell)+1} V_x\tilde\kappa_\pm$ in~\eqref{eq:interior-[1,1]} appeared. In order to cancel this term, we construct a profile
\begin{equation}
    \Phi_{1,\mathrm{tan}}^{\pm} = \Phi_{1,\mathrm{tan}}^{\pm}(t,x,m) \colon [0,T] \times \T \times \mathbb{N}_0 \to \R^9
\end{equation}
such that
\begin{equation}\label{eq:tangential-cancellation-minus}
    \mathscr L_{(0)}^{\pm}
    \Phi_{1,\mathrm{tan}}^{\pm}
    +
    \mathscr L_{(1)}^{\pm}
    \Phi_0^{\pm}
    =
    0.
\end{equation}
To this end, we first need some linear algebraic preparations.

\begin{remark}
    The tangential Knudsen layer equation~\eqref{eq:tangential-cancellation-minus} is analogous to the equation for the higher-order Knudsen layers in the kinetic theory of gases. In the asymptotic theory of the Boltzmann equation, the leading Knudsen layer solves a homogeneous kinetic half-space equation with the tangential variables frozen, whereas the equations for the higher-order Knudsen layers contain tangential derivatives of the lower-order Knudsen layers as inhomogeneous source terms; see for example~\cite[Chapter~3]{Sone2002}.
\end{remark}

For $\gamma\neq0$, define
\begin{equation}\label{eq:def-D-gamma}
    D_\gamma
    \coloneqq
    \operatorname{diag}
    \left(
        \gamma^{c_{0y}},
        \ldots,
        \gamma^{c_{8y}}
    \right).
\end{equation}
For any $\chi\in\R^9$, a simple calculation shows
\begin{equation}\label{eq:L0-gamma-chi}
    \mathscr L_{(0)}^-
    \left(
        \gamma^m\chi
    \right)
    =
    \gamma^m
    (D_\gamma-\mathscr C)\chi.
\end{equation}
Here, $\gamma^m \chi$ is considered as a function of $(t,x,m)$ when applying the operator $\mathscr{L}_{(0)}^{-}$. The next lemma gives some properties of the matrix $D_{\gamma}-\mathscr C$.

\begin{lemma}\label{lem:Dgamma-kernel}
    Suppose that $\gamma\in(-1,1)\setminus\{0\}$. Then
    \begin{equation}
        \ker(D_\gamma-\mathscr C)
        =
        \operatorname{span}\{ \tilde \kappa_- \},
        \qquad
        \operatorname{rank}(D_\gamma-\mathscr C)=8.
        \label{eq:Dgamma-kernel}
    \end{equation}
    The left kernel of $D_{\gamma}-\mathscr C$ is spanned by
    \begin{equation}\label{eq:lambda_minus}
        \lambda_-
        \coloneqq
        e_N-2e_{NE}-2e_{NW},
    \end{equation}
    and the space of $\chi \in \R^9$ solving the linear equation
    \begin{equation}\label{eq:chi-equation}
        (D_\gamma-\mathscr C)\chi
        =
        -\gamma V_x \tilde \kappa_-
    \end{equation}
    is a one-dimensional affine space.
\end{lemma}

\begin{proof}
    Let $f\in\ker(D_\gamma-\mathscr C)$ and write
    \[
        \rho \coloneqq \rho_f,
        \qquad
        j=(j_x,j_y) \coloneqq j_f.
    \]
    Then the condition that $f \in \ker(D_{\gamma}-\mathscr C)$ can be written as
    \begin{equation}\label{eq:kernel-componentwise}
        \left(
            \gamma^{c_{iy}}-\gamma
        \right)f_i
        =
        (1-\gamma)
        w_i
        \left(
            \rho+3c_i\cdot j
        \right), \qquad i \in \{ 0,\ldots,8 \}.
    \end{equation}
    For $i=N,NE,NW$, since $c_{iy}=1$, equations~\eqref{eq:kernel-componentwise} can be written as
    \[
        \rho+3j_y=0,
        \qquad
        \rho+3(j_x+j_y)=0,
        \qquad
        \rho+3(-j_x+j_y)=0,
    \]
    which gives
    \begin{equation}
        j_x=0,
        \qquad
        j_y=-\frac{\rho}{3}.
        \label{eq:kernel-moments-first}
    \end{equation}
    For $i=0,E,W$, by $c_{iy}=0$ and~\eqref{eq:kernel-moments-first}, equations~\eqref{eq:kernel-componentwise} can be written as
    \begin{equation}\label{eq:f0EW}
        f_0=w_0 \rho=\frac49\rho,
        \qquad
        f_E=w_E \rho=\frac19\rho, \qquad f_W=w_W \rho=\frac19\rho.
    \end{equation}
    For $i=S,SW,SE$, noting that $c_{iy}=-1$, equations~\eqref{eq:kernel-componentwise} can then be written as
    \begin{equation}
        f_i
        =
        \frac{2\gamma}{1+\gamma}
        w_i\rho.
    \end{equation}
    Thus, writing
    \[
        B
        \coloneqq
        f_S+f_{SW}+f_{SE},
    \]
    we have
    \begin{equation}
        B
        =
        \frac{\gamma}{3(1+\gamma)}\rho.
        \label{eq:B-down}
    \end{equation}
    On the other hand, by~\eqref{eq:kernel-moments-first} and~\eqref{eq:f0EW}, we have
    \begin{align}
        \begin{aligned}
            2B
            & =
            \left( f_0+f_E+f_W+f_N+f_{NE}+f_{NW}+f_S+f_{SW}+f_{SE} \right) \\
            & \qquad -\left( f_N+f_{NE}+f_{NW}-f_S-f_{SW}-f_{SE} \right) \\
            & \qquad -\left( f_0+f_E+f_W \right) \\
            & =\rho-j_y-\frac{2}{3}\rho=\frac{2}{3}\rho.
        \end{aligned}
    \end{align}
    Combining this with~\eqref{eq:B-down}, we obtain $\rho=0$; therefore $j=0$ as well by~\eqref{eq:kernel-moments-first}. Equations~\eqref{eq:kernel-componentwise} then imply that all $f_i=0$ except for $i=N,NE,NW$. Hence $\rho=0$ and $j=0$ yield
    \begin{equation}
        f_N+f_{NE}+f_{NW}=0, \qquad f_{NE}-f_{NW}=0.
    \end{equation}
    These relations give
    \begin{equation}
        f
        \in
        \operatorname{span}\{ \tilde \kappa_- \},
    \end{equation}
    which shows
    \begin{equation}
        \ker(D_{\gamma}-\mathscr C) \subset \operatorname{span}\{ \tilde \kappa_- \}.
    \end{equation}
    The other inclusion can be verified by a simple calculation and it proves the first assertion of~\eqref{eq:Dgamma-kernel}; the rank statement follows from this.

    Next, by~\eqref{eq:lambda_minus}, a direct calculation shows
    \begin{equation}\label{eq:lambda_D}
        \lambda_-^{T}
        D_\gamma
        =
        \gamma \lambda_{-}^{T}.
    \end{equation}
    On the other hand, for any $f \in \R^9$, we have
    \begin{align}\label{eq:lambda_Pi}
        \begin{aligned}
            \lambda_{-}^{T}\Pi f
            & =
            \lambda_{-}^{T}F^{\mathrm{eq}}(\rho_f,j_f) \\
            & =
            w_N\left( \rho_f+3(j_f)_y \right) \\
            & \qquad
            -2w_{NE}\left( \rho_f+3(j_f)_x+3(j_f)_y \right)-2w_{NW}\left( \rho_f-3(j_f)_x+3(j_f)_y \right) \\
            & =0.
        \end{aligned}
    \end{align}
    Hence
    \begin{equation}
        \lambda_{-}^{T}\Pi=0.
    \end{equation}
    By~\eqref{eq:lambda_D} and~\eqref{eq:lambda_Pi}, we obtain
    \begin{equation}
        \lambda_{-}^{T}(D_{\gamma}-\mathscr C)
        =
        \lambda_{-}^{T}D_{\gamma}-\lambda_{-}^{T}\left( \gamma I+(1-\gamma) \Pi \right)=0.
    \end{equation}
    Since the left kernel of $(D_{\gamma}-\mathscr C)$ is one-dimensional by the second equality in~\eqref{eq:Dgamma-kernel}, it is spanned by $\lambda_-$. Finally, a direct calculation shows
    \[
        \lambda_-\cdot V_x \tilde \kappa_-
        =
        0.
    \]
    Therefore, by the Fredholm alternative,~\eqref{eq:chi-equation} has a solution $\chi \in \R^9$, and the solution space is a one-dimensional affine space by the first equality in~\eqref{eq:Dgamma-kernel}.
\end{proof}

Let $\chi_{\gamma}^- \in \R^9$ be the unique solution to~\eqref{eq:chi-equation} with the normalizing condition
\begin{equation}\label{eq:chi-normalization}
    (\chi_\gamma^-)_N=0.
\end{equation}
The uniqueness follows from the first equality in~\eqref{eq:Dgamma-kernel} and $\tilde{\kappa}_{-,N}=2$. The next lemma gives some properties of $\chi_{\gamma}^-$.

\begin{lemma}\label{prop:chi-properties}
    Suppose that $\gamma \in (-1,1) \setminus \{ 0 \}$. The vector $\chi_\gamma^-$ satisfies
    \begin{gather}\label{eq:chi-symmetry}
        \begin{gathered}
        (\chi_\gamma^-)_0
        =
        (\chi_\gamma^-)_N
        =
        (\chi_\gamma^-)_S
        =
        0,
        \\
        (\chi_\gamma^-)_E
        =
        -(\chi_\gamma^-)_W,
        \\
        (\chi_\gamma^-)_{NE}
        =
        -(\chi_\gamma^-)_{NW},
        \\
        (\chi_\gamma^-)_{SW}
        =
        -(\chi_\gamma^-)_{SE}.
        \end{gathered}
    \end{gather}
    In particular,
    \begin{equation}\label{eq:chi-density-normal-momentum}
        \sum_{i=0}^8(\chi_\gamma^-)_i=0,
        \qquad
        \sum_{i=0}^8 c_{iy}(\chi_\gamma^-)_i=0.
    \end{equation}
    Furthermore, it holds that
    \begin{equation}\label{eq:chi-tangential-momentum}
        \sum_{i=0}^8 c_{ix}(\chi_\gamma^-)_i
        =
        -\frac{12\gamma}{1-\gamma}.
    \end{equation}
\end{lemma}

\begin{proof}
    Let $\mathscr R_x\colon\R^9\to\R^9$ denote the reflection with respect to the $y$-axis, namely
    \begin{equation}\label{eq:Rx}
        (\mathscr R_x f)_i=f_{i_x},
        \qquad
        c_{i_x}=(-c_{ix},c_{iy}).
    \end{equation}
    It is easy to see that both $D_\gamma$ and $\mathscr C$ commute with $\mathscr R_x$ and that
    \[
        \mathscr R_x
        \left(
            V_x \tilde \kappa_-
        \right)
        =
        -
        V_x \tilde \kappa_-.
    \]
    Hence, applying $\mathscr R_x$ to~\eqref{eq:chi-equation}, we obtain
    \begin{equation}
        (D_{\gamma}-\mathscr C)(-\mathscr R_x \chi_{\gamma}^{-})=-\gamma V_x \tilde \kappa_-.
    \end{equation}
    Since $(-\mathscr R_x \chi_{\gamma}^{-})_N=0$ by~\eqref{eq:chi-normalization}, from the uniqueness of the solution to~\eqref{eq:chi-equation} with the normalization condition~\eqref{eq:chi-normalization}, it follows that
    \begin{equation}
        \chi_{\gamma}^{-}=-\mathscr R_x \chi_{\gamma}^{-},
    \end{equation}
    which proves~\eqref{eq:chi-symmetry}. Identities~\eqref{eq:chi-density-normal-momentum} follow immediately from~\eqref{eq:chi-symmetry}.

    It remains to compute
    \[
        j_x^\chi
        \coloneqq
        \sum_{i=0}^8c_{ix}(\chi_\gamma^-)_i.
    \]
    Using~\eqref{eq:chi-density-normal-momentum}, we observe that
    \begin{equation}\label{eq:jxchi_1}
        (\Pi \chi_\gamma^-)_{NE}
        =
        3w_{NE} \cdot j_x^\chi
        =
        \frac1{12}j_x^\chi.
    \end{equation}
    Since $\gamma^{(c_{NE})_y}-\gamma=0$, the $NE$-component of~\eqref{eq:chi-equation} reads
    \begin{equation}\label{eq:jxchi_2}
        -(1-\gamma)(\Pi \chi_\gamma^-)_{NE}
        =
        (-\gamma V_x \tilde \kappa_-)_{NE}
        =
        \gamma.
    \end{equation}
    The identity~\eqref{eq:chi-tangential-momentum} then follows from~\eqref{eq:jxchi_1} and~\eqref{eq:jxchi_2}.
\end{proof}

With these preparations, we define the tangential Knudsen layer profile at the lower wall
\begin{equation}
    \Phi_{1,\mathrm{tan}}^{-}
    =
    \Phi_{1,\mathrm{tan}}^{-}(t,x,m)
    \colon
    [0,T] \times \T \times \mathbb{N}_0 \to \R^9
\end{equation}
by
\begin{equation}\label{eq:def-Phi1tan-minus}
    \Phi_{1,\mathrm{tan}}^-(t,x,m)
    \coloneqq
    -
    \partial_x\alpha_-(t,x)
    \gamma^m\chi_\gamma^-.
\end{equation}
For the upper wall, let $\mathscr R_y\colon\R^9\to\R^9$ denote the reflection with respect to the $x$-axis, namely
\begin{equation}
    (\mathscr R_y f)_i=f_{i_y},
    \qquad
    c_{i_y}=(c_{ix},-c_{iy}),
\end{equation}
and define
\begin{equation}
    \chi_\gamma^+
    \coloneqq
    \mathscr R_y\chi_\gamma^-,
    \qquad
    \Phi_{1,\mathrm{tan}}^+(t,x,m)
    \coloneqq
    -
    \partial_x\alpha_+(t,x)
    \gamma^m\chi_\gamma^+.
    \label{eq:def-Phi1tan-plus}
\end{equation}
We next verify that the identity~\eqref{eq:tangential-cancellation-minus} holds.

\begin{lemma}\label{lem:tangential-equation}
    Suppose that $\gamma \in (-1,1) \setminus \{ 0 \}$. Then we have the identity
    \begin{equation}
    \mathscr L_{(0)}^{\pm}
    \Phi_{1,\mathrm{tan}}^{\pm}
    +
    \mathscr L_{(1)}^{\pm}
    \Phi_0^{\pm}
    =
    0.
    \end{equation}
\end{lemma}

\begin{proof}
    For $\Phi_{0}^{-}$ and $\Phi_{1,\mathrm{tan}}^{-}$, this is a direct consequence of~\eqref{eq:L1-Phi0-minus},~\eqref{eq:L0-gamma-chi},~\eqref{eq:chi-equation}, and the definition of $\chi_{\gamma}^{-}$. For $\Phi_{0}^{+}$ and $\Phi_{1,\mathrm{tan}}^{+}$, we use the reflection $\mathscr R_y$. We also let $\mathscr R_y$ act on profiles in a pointwise manner. Since $\mathscr C$ commutes with $\mathscr R_y$, the definitions~\eqref{eq:L0-minus} and~\eqref{eq:L1-minus} of
    $\mathscr L_{(0)}^\pm$ and $\mathscr L_{(1)}^\pm$ give
    \begin{equation}\label{eq:L-reflection}
        \mathscr L_{(k)}^+
        \circ
        \mathscr R_y
        =
        \mathscr R_y
        \circ
        \mathscr L_{(k)}^-,
        \qquad
        k=0,1.
    \end{equation}
    Define the auxiliary lower-wall profiles
    \[
        \Psi_0^-(t,x,m)
        \coloneqq
        -\alpha_+(t,x)\gamma^m\tilde\kappa_-,
        \qquad
        \Psi_{1,\mathrm{tan}}^-(t,x,m)
        \coloneqq
        -\partial_x\alpha_+(t,x)\gamma^m\chi_\gamma^-.
    \]
    The lower-wall calculations above are unchanged with $\alpha_-$ replaced by $\alpha_+$, and therefore
    \begin{equation}\label{eq:Psi}
        \mathscr L_{(0)}^-
        \Psi_{1,\mathrm{tan}}^-
        +
        \mathscr L_{(1)}^-
        \Psi_0^-
        =
        0.
    \end{equation}
    Since
    \[
        \mathscr R_y\tilde\kappa_-=\tilde\kappa_+,
        \qquad
        \chi_\gamma^+=\mathscr R_y\chi_\gamma^-,
    \]
    we have
    \[
        \Phi_0^+
        =
        \mathscr R_y\Psi_0^-,
        \qquad
        \Phi_{1,\mathrm{tan}}^+
        =
        \mathscr R_y\Psi_{1,\mathrm{tan}}^-.
    \]
    Hence, by~\eqref{eq:L-reflection} and~\eqref{eq:Psi}, we obtain
    \begin{equation}
        \mathscr L_{(0)}^+
        \Phi_{1,\mathrm{tan}}^+
        +
        \mathscr L_{(1)}^+
        \Phi_0^+
        =
        \mathscr R_y
        \left(
            \mathscr L_{(0)}^-
            \Psi_{1,\mathrm{tan}}^-
            +
            \mathscr L_{(1)}^-
            \Psi_0^-
        \right)
        =0.
    \end{equation}
    This proves the assertion for both walls.
\end{proof}

Now, let
\begin{equation}\label{def:Phi1tan_real}
    \Phi_{1,\mathrm{tan},h}^\pm(t,z_{j,\ell})
    \coloneqq
    \Phi_{1,\mathrm{tan}}^\pm
    \left(
        t,jh,m_\pm(\ell)
    \right),
    \qquad
    (t,j,\ell) \in [0,T] \times \mathbb{Z}/M\mathbb{Z} \times \{ 0,\ldots,M-1 \}
\end{equation}
be the lattice realizations of $\Phi_{1,\mathrm{tan}}^{\pm}$. The next proposition gives the interior residual of $\Phi_{1,\mathrm{tan},h}^{\pm}$.

\begin{proposition}
\label{prop:interior-Phi1tan}
Suppose that $\gamma \in (-1,1) \setminus \{ 0 \}$. For $i \in \{0,\ldots,8\}$ and $z_{j,\ell} \in \Omega_h$ with $z_{j,\ell}+c_i h \in \Omega_h$, we have
\begin{equation}\label{eq:interior-Phi1tan}
    r_{h,i}[\Phi_{1,\mathrm{tan},h}^{\pm}]
    (t,z_{j,\ell})
    =
    \partial_x\alpha_\pm(t,jh)
    \gamma^{m_\pm(\ell)+1}
    \left(
        V_x\tilde\kappa_\pm
    \right)_i
    +
    O\left(
        h\Gamma^{m_\pm(\ell)}
    \right)
\end{equation}
uniformly for $0\leq t\leq T-h^2$.
\end{proposition}

\begin{proof}
    We prove the assertion for $\Phi_{1,\mathrm{tan},h}^{-}$; the claim for $\Phi_{1,\mathrm{tan},h}^{+}$ follows from a reflection argument similar to that in the proof of Lemma~\ref{lem:tangential-equation}. By~\eqref{eq:def-Phi1tan-minus} and~\eqref{def:Phi1tan_real}, we have
    \begin{equation}
        r_{h,i}[\Phi_{1,\mathrm{tan},h}^-](t,z_{j,\ell})
        =
        -\partial_x \alpha_-(t+h^2,jh+c_{ix} h)
        \gamma^{m_-(\ell)+c_{iy}}
        (\chi_\gamma^-)_i
        +
        \partial_x \alpha_-(t,jh)\gamma^{m_-(\ell)}
        (\mathscr C\chi_\gamma^-)_i.
    \end{equation}
    Hence
    \begin{align}
        r_{h,i}[\Phi_{1,\mathrm{tan},h}^-](t,z_{j,\ell})
        & =
        -\partial_x \alpha_-(t,jh)
        \gamma^{m_-(\ell)}
        \left(
            (D_\gamma-\mathscr C)\chi_\gamma^-
        \right)_i
        \\
        & \qquad
        -
        \left\{
            \partial_x \alpha_-(t+h^2,jh+c_{ix} h)-\partial_x \alpha_-(t,jh)
        \right\}
        \gamma^{m_-(\ell)}
        (D_\gamma\chi_\gamma^-)_i.
    \end{align}
    By~\eqref{eq:chi-equation} and the definition of $\chi_{\gamma}^{-}$, the first term on the right-hand side equals
    \[
        \partial_x\alpha_-(t,jh)
        \gamma^{m_-(\ell)+1}
        \left(
            V_x\tilde\kappa_-
        \right)_i.
    \]
    On the other hand, Taylor's theorem gives
    \[
        \partial_x \alpha_-(t+h^2,jh+c_{ix}h)-\partial_x \alpha_-(t,jh)=O(h).
    \]
    Combining these proves~\eqref{eq:interior-Phi1tan} for $\Phi_{1,\mathrm{tan},h}^{-}$.
\end{proof}

Using the tangential Knudsen layer correctors $\Phi_{1,\mathrm{tan},h}^{\pm}$, we refine the prediction function $\hat{f}_h^{[1,1]}$ further to
\begin{align}\label{eq:def-f111}
    \begin{aligned}
        \hat{f}_{h}^{[1,1,1]}
        & \coloneqq
        \hat f_h^{[1,1]}
        +
        h^4
        \left(
            \Phi_{1,\mathrm{tan},h}^{-}+\Phi_{1,\mathrm{tan},h}^{+}
        \right)
        \\
        & =
        \mathcal P_h[u+h^2 v_0,p+h^2 \pi_0]
        +
        h^3
        \left(
            \Phi_{0,h}^-+\Phi_{0,h}^+
        \right)
        +
        h^4
        \left(
            \Phi_{1,\mathrm{tan},h}^{-}+\Phi_{1,\mathrm{tan},h}^{+}
        \right)
    \end{aligned}
\end{align}
when $\gamma \in (-1,1) \setminus \{ 0 \}$. In the case of $\gamma=0$, we do not need an update, and we set
\begin{equation}
    \hat{f}_{h}^{[1,1,1]}
    \coloneqq
    \hat f_h^{[1,1]}.
\end{equation}
Applying Proposition~\ref{prop:interior-Phi1tan}, we obtain the following proposition.

\begin{proposition}\label{prop:asymptotics-[1,1,1]}
    Let
    \begin{equation}
        \hat r_h^{[1,1,1]}
        \coloneqq
        r_h[\hat f_h^{[1,1,1]}],
        \qquad
        \widehat R_h^{\pm,[1,1,1]}
        \coloneqq
        R_h^\pm[\hat f_h^{[1,1,1]}].
    \end{equation}
    For the interior residual $\hat{r}_{h}^{[1,1,1]}$, we have
    \begin{equation}\label{eq:interior-[1,1,1]}
        \hat r_{h}^{[1,1,1]}(t,z_{j,\ell})
        =
        O(h^6)
        +
        \sum_{\sigma=+,-}O(h^5 \Gamma^{m_{\sigma}(\ell)})
    \end{equation}
    uniformly in $(t,z_{j,\ell}) \in [0,T-h^2] \times \Omega_h$. In particular,
    \begin{equation}\label{eq:mass-interior-[1,1,1]}
        h^2
        \sum_{z\in\Omega_h}
        \sum_{i=0}^8
        \hat r_{h,i}^{[1,1,1]}(t,z)
        =O(h^6)
    \end{equation}
    uniformly in $t \in [0,T-h^2]$. Moreover, for the boundary residuals $\widehat{R}_{h}^{\pm,[1,1,1]}$, there exist smooth functions
    \[
        R^{\pm,[1,1,1],(4)}
        \colon
        [0,T] \times \T \to \R^3
    \]
    independent of $h$ such that
    \begin{equation}\label{eq:boundary-[1,1,1]}
        \widehat R_h^{\pm,[1,1,1]}
        =
        h^4R^{\pm,[1,1,1],(4)}
        +
        O(h^5)
    \end{equation}
    uniformly in $[0,T-h^2] \times \T_h$.
\end{proposition}

\begin{proof}
    The case of $\gamma=0$ is clearly easier and we shall assume $\gamma \neq 0$. By~\eqref{eq:def-f111} and the linearity of the residual operator $r_h$, we have
    \begin{equation}
        \hat r_h^{[1,1,1]}
        =
        \hat r_h^{[1,1]}
        +
        h^4
        \left(
            r_h[\Phi_{1,\mathrm{tan},h}^-]
            +
            r_h[\Phi_{1,\mathrm{tan},h}^+]
        \right).
    \end{equation}
    The leading $O(h^4)$ terms in~\eqref{eq:interior-[1,1]} are canceled by the leading term in~\eqref{eq:interior-Phi1tan}. This yields~\eqref{eq:interior-[1,1,1]}. Summing~\eqref{eq:interior-[1,1,1]} over the lattice sites $z \in \Omega_h$ and using
    \[
        \sum_{m=0}^{M-1}\Gamma^m\leq\frac1{1-\Gamma}
    \]
    gives~\eqref{eq:mass-interior-[1,1,1]}.
    
    For the boundary residuals, noting that $\Phi_{1,\mathrm{tan},h}^{\pm}$ appear as $O(h^4)$ terms in~\eqref{eq:def-f111}, the bound~\eqref{eq:boundary-[1,1]} and Taylor's theorem imply the existence of smooth functions
    \[
        R^{\pm,[1,1,1],(4)}
        \colon
        [0,T] \times \T \to \R^3
    \]
    such that
    \[
        \widehat R_h^{\pm,[1,1,1]}
        =
        h^4 R^{\pm,[1,1,1],(4)}
        +
        O(h^5).
    \]
    This proves~\eqref{eq:boundary-[1,1,1]}.
\end{proof}

\subsection{Second Stokes and Knudsen layer correctors}\label{sec:second-correctors}

The interior residual and the boundary residuals of $\hat{f}_{h}^{[1,1,1]}$ are
\[
    O(h^6)+\sum_{\sigma=+,-}O(h^5 \Gamma^{m_{\sigma}(\ell)})
\]
and $O(h^4)$, respectively, by~\eqref{eq:interior-[1,1,1]} and~\eqref{eq:boundary-[1,1,1]}. To prove Theorem~\ref{thm:optimal_convergence}, we need to further remove the $O(h^4)$ term in~\eqref{eq:boundary-[1,1,1]}. This will be done in the current section. However, we do not need new constructions and we only need to repeat the constructions of the Stokes and the Knudsen layer correctors in previous sections (the tangential Knudsen layer correctors are unnecessary).

We first repeat the construction of the Stokes corrector. Using again the direct-sum decomposition~\eqref{eq:boundary-direct-sum}, define uniquely
\[
    g_1^\pm \colon [0,T] \times \T \to \R^2,
    \qquad
    \beta_\pm \colon [0,T] \times \T \to \R
\]
by the equations
\begin{equation}\label{eq:def-g1-beta}
    R^{\pm,[1,1,1],(4)}
    +
    \mathcal C_\pm g_1^\pm
    =
    \beta_\pm \kappa_\pm,
\end{equation}
where $R^{\pm,[1,1,1],(4)}$ are the smooth functions in~\eqref{eq:boundary-[1,1,1]}. We check that $g_{1}^{\pm}$ satisfy the necessary admissibility condition.

\begin{lemma}\label{lem:second-flux-compatibility}
    The functions $g_1^\pm$ satisfy
    \begin{equation}\label{eq:second-flux-compatibility}
        -\int_{\T}
            (g_1^-)_y(t,x) \dd x
        +
        \int_{\T}
            (g_1^+)_y(t,x) \dd x
        =
        0
    \end{equation}
    for every $t\in[0,T]$.
\end{lemma}

\begin{proof}
    Note that the Knudsen layer corrections introduced so far have zero density by~\eqref{eq:kappa-zero-moments} and \eqref{eq:chi-density-normal-momentum}, while the pressure $\pi_0$ is normalized to have zero mean. Moreover, we have
    \begin{equation}
        M_h[\hat{r}_{h}^{[1,1,1]}(t)]
        =
        O(h^6)
    \end{equation}
    by~\eqref{eq:mass-interior-[1,1,1]}. Here, the operator $M_h$ is defined by~\eqref{eq:operator_mass}. Using these properties, an argument similar to the proof of Lemma~\ref{lem:first-flux-compatibility} applies to $\hat f_h^{[1,1,1]}$ as well. The details are omitted.
\end{proof}

Let $v_{1,\mathrm{in}}$ be a smooth divergence-free vector that is periodic in $x$. Taking $g_{1}^{\pm}$ as the boundary data, let $(v_1,\pi_1)$ be the solution to the Stokes equations
\begin{equation}\label{eq:second-Stokes-corrector}
\begin{dcases}
    \partial_t v_1+\nabla\pi_1
    =
    \nu\Delta v_1
    &\text{in $(0,T)\times\Omega$},
    \\
    \nabla\cdot v_1=0
    &\text{in $(0,T)\times\Omega$},
    \\
    v_1(t,x,0)=g_1^-(t,x)
    &\text{for $(t,x)\in(0,T)\times\T$},
    \\
    v_1(t,x,1)=g_1^+(t,x)
    &\text{for $(t,x)\in(0,T)\times\T$},
    \\
    v_1(0,\cdot)=v_{1,\mathrm{in}}
    &\text{in $\Omega$}
\end{dcases}
\end{equation}
with periodicity in $x$ and normalization $\int_\Omega\pi_1(t,x,y)\dd x\dd y=0$. The necessary admissibility condition is guaranteed by Lemma~\ref{lem:second-flux-compatibility}. The initial data $v_{1,\mathrm{in}}$ is required to satisfy the standard compatibility conditions of sufficiently high order. Set
\begin{equation}\label{eq:second-outer-correction}
    u^{[2]}
    \coloneqq
    u+h^2v_0+h^3v_1,
    \qquad
    p^{[2]}
    \coloneqq
    p+h^2\pi_0+h^3\pi_1.
\end{equation}
Then let
\begin{align}\label{eq:def-f211}
    \begin{aligned}
        \hat{f}_{h}^{[2,1,1]}
        & \coloneqq
        \mathcal P_h[u^{[2]},p^{[2]}]
        +
        h^3
        \left(
            \Phi_{0,h}^-+\Phi_{0,h}^+
        \right)
        +
        h^4
        \left(
            \Phi_{1,\mathrm{tan},h}^-
            +
            \Phi_{1,\mathrm{tan},h}^+
        \right)
        \\
        & =
        \hat{f}_{h}^{[1,1,1]}
        +
        \left(
            \mathcal P_h[u^{[2]},p^{[2]}]
            -
            \mathcal P_h[u^{[1]},p^{[1]}]
        \right).
    \end{aligned}
\end{align}
when $\gamma \in (-1,1) \setminus \{ 0 \}$. In the case of $\gamma=0$, let
\begin{align}\label{eq:def-f211-gamma_0}
    \begin{aligned}
        \hat{f}_{h}^{[2,1,1]}
        & \coloneqq
        \mathcal P_h[u^{[2]},p^{[2]}]
        +
        h^3
        \left(
            \Phi_{0,h}^-+\Phi_{0,h}^+
        \right)
        \\
        & =
        \hat{f}_{h}^{[1,1,1]}
        +
        \left(
            \mathcal P_h[u^{[2]},p^{[2]}]
            -
            \mathcal P_h[u^{[1]},p^{[1]}]
        \right).
    \end{aligned}
\end{align}
Applying Lemma~\ref{lem:pred_error} and Lemma~\ref{lem:pred_error_inhom}, we obtain the following proposition.

\begin{proposition}\label{prop:asymptotics-[2,1,1]}
    Let
    \begin{equation}
        \hat r_h^{[2,1,1]}
        \coloneqq
        r_h[\hat f_h^{[2,1,1]}],
        \qquad
        \widehat R_h^{\pm,[2,1,1]}
        \coloneqq
        R_h^\pm[\hat f_h^{[2,1,1]}].
    \end{equation}
    For the interior residual $\hat r_h^{[2,1,1]}$, we have
    \begin{equation}\label{eq:interior-[2,1,1]}
        \hat r_h^{[2,1,1]}(t,z_{j,\ell})
        =
        O(h^6)
        +
        \sum_{\sigma=+,-}
        O\left(
            h^5\Gamma^{m_\sigma(\ell)}
        \right)
    \end{equation}
    uniformly in $(t,z_{j,\ell}) \in [0,T-h^2] \times \Omega_h$. Moreover, the boundary residuals $\widehat{R}_{h}^{\pm,[2,1,1]}$ satisfy
    \begin{equation}\label{eq:boundary-[2,1,1]}
        \widehat R_h^{\pm,[2,1,1]}
        =
        h^4\beta_\pm\kappa_\pm
        +
        O(h^5)
    \end{equation}
    uniformly in $[0,T-h^2] \times \T_h$, where $\beta_{\pm}$ are defined by~\eqref{eq:def-g1-beta}.
\end{proposition}

\begin{proof}
    By~\eqref{eq:def-f211} or~\eqref{eq:def-f211-gamma_0}, we have
    \begin{equation}\label{eq:f211-minus-f111}
        \hat f_h^{[2,1,1]}
        -
        \hat f_h^{[1,1,1]}
        =
        \mathcal P_h[u^{[2]},p^{[2]}]
        -
        \mathcal P_h[u^{[1]},p^{[1]}].
    \end{equation}
    Since
    \[
        u^{[2]}-u^{[1]}=h^3v_1,
        \qquad
        p^{[2]}-p^{[1]}=h^3\pi_1,
    \]
    we have
    \begin{equation}\label{eq:f211-affine-difference}
        \hat f_h^{[2,1,1]}
        -
        \hat f_h^{[1,1,1]}
        =
        h^3
        \left(
            \mathcal P_h[v_1,\pi_1]
            -
            F^{\mathrm{eq}}(1,0)
        \right).
    \end{equation}
    Since the constant equilibrium $F^{\mathrm{eq}}(1,0)$ has zero interior and boundary residuals, the linearity of the residual operators $r_h$ and $R_{h}^{\pm}$ gives
    \begin{align}
        \hat r_h^{[2,1,1]}
        -
        \hat r_h^{[1,1,1]}
        &=
        h^3 r_h[\mathcal{P}_h[v_1,\pi_1]],
        \label{eq:interior-difference-[2,1,1]}
        \\
        \widehat R_h^{\pm,[2,1,1]}
        -
        \widehat R_h^{\pm,[1,1,1]}
        &=
        h^3 R_h^\pm[\mathcal{P}_h[v_1,\pi_1]].
        \label{eq:boundary-difference-[2,1,1]}
    \end{align}
    By Lemma~\ref{lem:pred_error_inhom}, we have
    \[
        r_h [\mathcal{P}_h[v_1,\pi_1]]=O(h^6),
    \]
    which implies
    \[
        \hat r_h^{[2,1,1]}
        =
        \hat r_h^{[1,1,1]}
        +
        O(h^9).
    \]
    Combining this with~\eqref{eq:interior-[1,1,1]}, we obtain~\eqref{eq:interior-[2,1,1]}. On the other hand, Lemma~\ref{lem:pred_error_inhom} yields
    \begin{equation}
        R_h^\pm [\mathcal{P}_h[v_1,\pi_1]]
        =
        h\mathcal C_\pm g_1^\pm
        +
        O(h^2),
    \end{equation}
    which implies
    \begin{equation}
        \widehat{R}_h^{\pm,[2,1,1]}
        =
        \widehat{R}_h^{\pm,[1,1,1]}
        +
        h^4 \mathcal{C}_{\pm} g_{1}^{\pm}
        +
        O(h^5).
    \end{equation}
    Combining this with~\eqref{eq:boundary-[1,1,1]}, we obtain
    \begin{align}
        \widehat R_h^{\pm,[2,1,1]}
        &=
        h^4
        \left(
            R^{\pm,[1,1,1],(4)}
            +
            \mathcal C_\pm g_1^\pm
        \right)
        +
        O(h^5)
        \\
        &=
        h^4\beta_\pm\kappa_\pm
        +
        O(h^5),
    \end{align}
    where the last equality follows from~\eqref{eq:def-g1-beta}. This proves~\eqref{eq:boundary-[2,1,1]}.
\end{proof}

We next repeat the construction of the Knudsen layer correctors. Let
\begin{equation}\label{eq:def-Phi1-pm}
    \Phi_1^\pm(t,x,m)
    \coloneqq
    -\beta_\pm(t,x)\gamma^m\tilde\kappa_\pm,
    \qquad
    (t,x,m)\in[0,T]\times\T\times\mathbb{N}_0.
\end{equation}
We remind the reader that we are using the convention $0^0=1$. Their lattice realizations are defined by
\begin{equation}\label{eq:def-Phi1h-pm}
    \Phi_{1,h}^\pm(t,z_{j,\ell})
    \coloneqq
    \Phi_1^\pm
    \left(
        t,jh,m_\pm(\ell)
    \right),
    \qquad
    (t,j,\ell) \in [0,T] \times \mathbb{Z}/M\mathbb{Z} \times \{ 0,\ldots,M-1 \}.
\end{equation}
We have
\begin{equation}
    \left|
        \Phi_{1,h}^\pm(t,z_{j,\ell})
    \right|
    \leq
    C_T\Gamma^{m_\pm(\ell)}
\end{equation}
for some constant $C_T>0$ depending only on $T$. The following propositions give the interior and the boundary residuals of $\Phi_{1,h}^{\pm}$. We omit their proofs since they are identical to those of Proposition~\ref{prop:boundary-Phi0} and Proposition~\ref{prop:interior-Phi0}.

\begin{proposition}
\label{prop:boundary-Phi1}
We have
\begin{equation}\label{eq:boundary-Phi1}
    R_h^\pm[\Phi_{1,h}^\pm]
    =
    -\beta_\pm\kappa_\pm
    +
    O(h^2)
\end{equation}
uniformly in $[0,T-h^2] \times \T_h$.
\end{proposition}

\begin{proposition}
\label{prop:interior-Phi1}
For $i \in \{0,\ldots,8\}$ and $z_{j,\ell}\in\Omega_h$ with $z_{j,\ell}+c_i h \in \Omega_h$, we have
\begin{align}
\label{eq:interior-Phi1}
    \begin{aligned}
        r_{h,i}[\Phi_{1,h}^\pm](t,z_{j,\ell})
        &=
        -
        \left\{
            \beta_\pm
            \left(
                t+h^2,jh+c_{ix}h
            \right)
            -
            \beta_\pm(t,jh)
        \right\}
        \gamma^{m_\pm(\ell)+1}
        \tilde\kappa_{\pm,i}
        \\
        &=
        -h\partial_x\beta_\pm(t,jh)
        \gamma^{m_\pm(\ell)+1}
        \left(
            V_x\tilde\kappa_\pm
        \right)_i
        +
        O\left(
            h^2\Gamma^{m_\pm(\ell)+1}
        \right)
    \end{aligned}
\end{align}
uniformly for $0\leq t\leq T-h^2$. Here, $V_x$ is defined by~\eqref{eq:def-Vx}.
\end{proposition}

Using the correctors $\Phi_{1,h}^\pm$, we finally define the prediction function that we shall use in the proof of Theorem~\ref{thm:optimal_convergence} as follows:
\begin{align}\label{eq:corrected_pred_final}
    \begin{aligned}
        \hat f_h^{[2,2,1]}
        &\coloneqq
        \hat f_h^{[2,1,1]}
        +
        h^4
        \left(
            \Phi_{1,h}^-+\Phi_{1,h}^+
        \right)
        \\
        &=
        \mathcal P_h[u^{[2]},p^{[2]}]
        +
        h^3
        \left(
            \Phi_{0,h}^-+\Phi_{0,h}^+
        \right)
        \\
        &\qquad
        +
        h^4
        \left(
            \Phi_{1,\mathrm{tan},h}^-
            +
            \Phi_{1,\mathrm{tan},h}^+
            +
            \Phi_{1,h}^-
            +
            \Phi_{1,h}^+
        \right)
    \end{aligned}
\end{align}
when $\gamma \in (-1,1) \setminus \{ 0 \}$. In the case of $\gamma=0$, let
\begin{align}\label{eq:corrected_pred_final_gamma-0}
    \begin{aligned}
        \hat f_h^{[2,2,1]}
        &\coloneqq
        \hat f_h^{[2,1,1]}
        +
        h^4
        \left(
            \Phi_{1,h}^-+\Phi_{1,h}^+
        \right)
        \\
        &=
        \mathcal P_h[u^{[2]},p^{[2]}]
        +
        h^3
        \left(
            \Phi_{0,h}^-+\Phi_{0,h}^+
        \right)
        +
        h^4
        \left(
            \Phi_{1,h}^-
            +
            \Phi_{1,h}^+
        \right).
    \end{aligned}
\end{align}
Applying Propositions~\ref{prop:asymptotics-[2,1,1]}--\ref{prop:interior-Phi1} to~\eqref{eq:corrected_pred_final} or~\eqref{eq:corrected_pred_final_gamma-0}, we immediately obtain the following proposition. Note that the contribution of $h^4 \Phi_{1,h}^{\mp}$ to $\widehat{R}_{h}^{\pm,[2,1,1]}$ is $O(h^4 \Gamma^{M-1})$ and is exponentially small (note that $M=h^{-1}$).

\begin{proposition}\label{prop:asymptotics-[2,2,1]}
Let
\begin{equation}
    \hat r_h^{[2,2,1]}
    \coloneqq
    r_h[\hat f_h^{[2,2,1]}],
    \qquad
    \widehat R_h^{\pm,[2,2,1]}
    \coloneqq
    R_h^\pm[\hat f_h^{[2,2,1]}].
\end{equation}
Then the interior residual $\hat{r}_{h}^{[2,2,1]}$ satisfies
\begin{equation}\label{eq:interior-[2,2,1]}
    \hat r_h^{[2,2,1]}(t,z_{j,\ell})
    =
    O(h^6)
    +
    \sum_{\sigma=+,-}
    O\left(
        h^5\Gamma^{m_\sigma(\ell)}
    \right)
\end{equation}
uniformly in $(t,z_{j,\ell}) \in [0,T-h^2] \times \Omega_h$. Moreover, the boundary residuals $\widehat{R}_{h}^{\pm,[2,2,1]}$ satisfy
\begin{equation}\label{eq:boundary-[2,2,1]}
    \widehat R_h^{\pm,[2,2,1]}
    =
    O(h^5)
\end{equation}
uniformly in $[0,T-h^2] \times \T_h$.
\end{proposition}

\section{Proof of the optimal error estimates}\label{sec:convergence}

Having defined the prediction function $\hat{f}_{h}^{[2,2,1]}$, we are now ready to prove Theorem~\ref{thm:optimal_convergence}. We first prove the following.

\begin{proposition}\label{prop:refined_error}
    There exists a constant $C_T>0$ depending only on $T$ such that the prediction function $\hat{f}_h^{[2,2,1]}$ satisfies
    \begin{equation}\label{eq:residual_final}
        \widetilde{S}\hat{f}_{h}^{[2,2,1]}(t_{n+1})
        =
        \mathscr{C}\hat{f}_{h}^{[2,2,1]}(t_n)
        +
        \hat{r}_{h}^{[2,2,1],n}
        +
        \widehat{R}_{h}^{[2,2,1],n}
    \end{equation}
    for some $\hat{r}_{h}^{[2,2,1],n},\widehat{R}_{h}^{[2,2,1],n} \in X_h$ satisfying for $z_{j,\ell} \in \Omega_h$ and $0 \leq n \leq N_h-1$,
    \begin{equation}\label{eq:rR_bounds_final}
        \| \hat{r}_{h}^{[2,2,1],n}(z_{j,\ell}) \| \leq C_T
        \left(
            h^6+h^5 \sum_{\sigma=+,-}\Gamma^{m_{\sigma}(\ell)}
        \right),
        \qquad
        \| \widehat{R}_{h}^{[2,2,1],n}(z_{j,\ell}) \| \leq C_T h^{5}
    \end{equation}
    and
    \begin{equation}\label{eq:support_final}
        \widehat{R}_{h,i}^{[2,2,1],n}(z_{j,\ell})=0 \qquad (\text{for $i \in \{ 0,\ldots,8 \}$ with $z_{j,\ell}+c_i h \in \Omega_h$}).
    \end{equation}
    In particular,
    \begin{equation}\label{eq:final_residual_order_sum}
        \| \hat{r}_{h}^{[2,2,1],n} \|_{X_h}
        +
        \| \widehat{R}_{h}^{[2,2,1],n} \|_{X_h} \leq C_T h^{11/2}.
    \end{equation}
\end{proposition}

\begin{proof}
    Let
    \[
        \hat r_h^{[2,2,1],n}
        \coloneqq
        r_h[\hat{f}_{h}^{[2,2,1]}](t_n),
        \qquad
        \widehat R_h^{[2,2,1],n}
        \coloneqq
        R_h[\hat f_h^{[2,2,1]}](t_n),
    \] 
    where the operators $r_h$ and $R_h$ are defined by~\eqref{def:interior-residual} and~\eqref{def:boundary-residual}. Then~\eqref{eq:residual_final} follows from~\eqref{def:E} and~\eqref{eq:E_rR}. The support condition~\eqref{eq:support_final} is automatic from the definition~\eqref{def:boundary-residual}. The first inequality in~\eqref{eq:rR_bounds_final} follows from~\eqref{eq:interior-[2,2,1]}. Since $M=h^{-1}$ and $0 \leq \Gamma<1$, we have
    \[
        \sum_{\ell=0}^{M-1}\Gamma^{2m_\pm(\ell)}
        \leq
        \frac{1}{1-\Gamma^2}.
    \]
    Hence
    \begin{align}\label{eq:final_residual_r}
        \|
            \hat r_h^{[2,2,1],n}
        \|_{X_h}
        & \leq
        C_T h^6
        +
        C_T h^5
        \left(
            h^2 M
            \sum_{\ell=0}^{M-1}
            \Gamma^{2\ell}
        \right)^{1/2}
        \\
        & \leq
        C_T h^{11/2}.
    \end{align}
    On the other hand, the second inequality in~\eqref{eq:rR_bounds_final} follows from~\eqref{eq:boundary-[2,2,1]}. Since $\widehat{R}_h^{[2,2,1]}$ is supported only on the first and the last lattice layers, which contain $O(h^{-1})$ lattice sites, we obtain
    \begin{align}\label{eq:final_residual_R}
        \|
            \widehat R_h^{[2,2,1],n}
        \|_{X_h}
        \leq
        C_T h^{11/2}.
    \end{align}
    Combining~\eqref{eq:final_residual_r} and~\eqref{eq:final_residual_R}, we obtain~\eqref{eq:final_residual_order_sum}.
\end{proof}

Combining Proposition~\ref{prop:refined_error} and the $L^2$-stability estimate in Proposition~\ref{prop:stability}, we obtain the following.

\begin{proposition}\label{prop:population-error-refined}
    Let $f_h^n$ be the solution to the LBM scheme
    \eqref{eq:LBM-compact}. Then there exists a constant $C_T>0$ depending only on $T$ such that
    \begin{equation}\label{eq:population-error-refined}
        \|
            f_h^n-\hat f_h^{[2,2,1]}(t_n)
        \|_{X_h}
        \leq
        \|
            f_h^0-\hat f_h^{[2,2,1]}(0)
        \|_{X_h}
        +
        C_T h^{7/2}
    \end{equation}
    for $0\leq n\leq N_h$. In particular, if $f_h^0$ is a third-order initialization in the sense of Definition~\ref{def:3rd-initialization}, then
    \begin{equation}\label{eq:population-error-refined-h3}
        \|
            f_h^n-\hat f_h^{[2,2,1]}(t_n)
        \|_{X_h}
        \leq
        C_T h^3
    \end{equation}
    for $0\leq n\leq N_h$.
\end{proposition}

\begin{proof}
    Define $e_{h}^{n} \in X_h$ for $0 \leq n \leq N_h$ by
    \begin{equation}
        e_{h}^{n}
        \coloneqq
        f_{h}^{n}-\hat{f}_{h}^{[2,2,1]}(t_n).
    \end{equation}
    By abuse of notation, we shall use the same symbol $e_{h}^{n}$ as in the proof of Theorem~\ref{thm:Junk_convergence}. This shall cause no confusion. Recall that we write
    \begin{equation}
        A_h
        \coloneqq
        \widetilde{S}^{-1} \mathscr C.
    \end{equation}
    By Proposition~\ref{prop:stability}, we have
    \begin{equation}\label{eq:stability_our_theorem}
        \| \widetilde{S}^{-1} \|_{X_h \to X_h}=1,
        \qquad
        \| A_h \|_{X_h \to X_h} \leq 1.
    \end{equation}
    By~\eqref{eq:LBM-compact} and~\eqref{eq:residual_final}, we have
    \begin{equation}
        e_{h}^{n+1}=A_h e_{h}^{n}
        -
        \widetilde S^{-1} \hat{r}_{h}^{[2,2,1],n}
        -
        \widetilde S^{-1} \widehat{R}_{h}^{[2,2,1],n}.
    \end{equation}
    Hence by Duhamel's principle,
    \begin{equation}
        e_{h}^{n}
        =
        (A_h)^n e_{h}^{0}
        -
        \sum_{k=0}^{n-1}(A_{h})^{n-k-1}\widetilde{S}^{-1}\hat{r}_{h}^{[2,2,1],k}
        -
        \sum_{k=0}^{n-1}(A_{h})^{n-k-1}\widetilde{S}^{-1}\widehat{R}_{h}^{[2,2,1],k}
    \end{equation}
    for $1 \leq n \leq N_h$. From~\eqref{eq:final_residual_order_sum} and~\eqref{eq:stability_our_theorem}, it follows that
    \begin{align}
        \| e_{h}^{n} \|_{X_h}
        & \leq
        \| e_{h}^{0} \|_{X_h}
        +
        \sum_{k=0}^{n-1}\| \hat{r}_{h}^{[2,2,1],k} \|_{X_h}
        +
        \sum_{k=0}^{n-1}\| \widehat{R}_{h}^{[2,2,1],k} \|_{X_h} \\
        & \leq 
        \| f_{h}^{0}-\hat{f}_{h}^{[2,2,1]}(0) \|_{X_h}
        +
        \sum_{k=0}^{n-1}C_T h^{11/2} \\
        & \leq
        \| f_{h}^{0}-\hat{f}_{h}^{[2,2,1]}(0) \|_{X_h}
        +C_T h^{7/2}.
    \end{align}
    This proves~\eqref{eq:population-error-refined}.
    
    It remains to estimate the initial error. Let us assume $\gamma \neq 0$ as the case of $\gamma=0$ is easier. By~\eqref{eq:corrected_pred_final}, we have
    \begin{align}
        \hat{f}_{h}^{[2,2,1]}
        -
        \hat{f}_{h}
        & =
        \mathcal P_h[u^{[2]},p^{[2]}] - \mathcal P_h[u,p]
        +
        h^3
        \left(
            \Phi_{0,h}^-+\Phi_{0,h}^+
        \right)
        \\
        & \qquad
        +
        h^4
        \left(
            \Phi_{1,\mathrm{tan},h}^-
            +
            \Phi_{1,\mathrm{tan},h}^+
            +
            \Phi_{1,h}^-
            +
            \Phi_{1,h}^+
        \right)
        \\
        & =
        h^2 \left( \mathcal P_h[v_0+hv_1,\pi_0+h\pi_1] - F^{\mathrm{eq}}(1,0) \right)
        +
        h^3
        \left(
            \Phi_{0,h}^-+\Phi_{0,h}^+
        \right)
        \\
        & \qquad
        +
        h^4
        \left(
            \Phi_{1,\mathrm{tan},h}^-
            +
            \Phi_{1,\mathrm{tan},h}^+
            +
            \Phi_{1,h}^-
            +
            \Phi_{1,h}^+
        \right).
    \end{align}
    Hence, noting that the prediction operator $\mathcal{P}_h$ is defined by~\eqref{eq:Junk_prediction} and~\eqref{def:P}, we obtain
    \begin{equation}
        \| \hat{f}_{h}^{[2,2,1]}(0) - \hat{f}_{h}(0) \|_{X_h}
        \leq
        C_T h^3.
    \end{equation}
    Therefore, if $f_{h}^{0}$ is a third-order initialization,
    \begin{equation}
        \| f_{h}^{0}-\hat{f}_{h}^{[2,2,1]}(0) \|_{X_h}
        \leq 
        \| f_{h}^{0}-\hat{f}_{h}(0) \|_{X_h}
        +
        \| \hat{f}_{h}(0)-\hat{f}_{h}^{[2,2,1]}(0) \|_{X_h}
        \leq
        C_T h^3.
    \end{equation}
    Combining this with~\eqref{eq:population-error-refined} proves~\eqref{eq:population-error-refined-h3}.
\end{proof}

We finally prove Theorem~\ref{thm:optimal_convergence}.

\begin{proof}[Proof of Theorem~\ref{thm:optimal_convergence}]
    Set
    \begin{equation}
        \widehat{U}_{h}^{[2,2,1]}
        \coloneqq
        \frac{1}{h}j_{\hat{f}_{h}^{[2,2,1]}},
        \qquad
        \widehat{P}_{h}^{[2,2,1]}
        \coloneqq
        \frac{1}{3h^2}
        \left(
            \rho_{\hat{f}_{h}^{[2,2,1]}} - 1
        \right).
    \end{equation}
    Let us assume $\gamma \neq 0$ as the case of $\gamma = 0$ is easier. By~\eqref{eq:corrected_pred_final}, we have
    \begin{equation}
        \widehat{U}_{h}^{[2,2,1]}
        =
        \frac{1}{h} j_{\mathcal{P}_{h}[u^{[2]},p^{[2]}]}
        +
        O(h^2),
        \qquad
        \widehat{P}_{h}^{[2,2,1]}
        =
        \frac{1}{3h^2}
        \left(
            \rho_{\mathcal{P}_{h}[u^{[2]},p^{[2]}]} - 1
        \right)
        +
        O(h).
    \end{equation}
    Applying~\eqref{eq:Ph_moments} to the pair $(u^{[2]},p^{[2]})$, we obtain
    \begin{align}
        \frac{1}{h} j_{\mathcal{P}_{h}[u^{[2]},p^{[2]}]}
        & =
        u^{[2]}
        +
        O(h^2)
        =
        u
        +
        O(h^2),
        \\
        \frac{1}{3h^2}
        \left(
            \rho_{\mathcal{P}_{h}[u^{[2]},p^{[2]}]} - 1
        \right)
        & =
        p^{[2]}
        +
        O(h^2)
        =
        p
        +
        O(h^2).
    \end{align}
    From these, it follows that
    \begin{equation}\label{eq:corrected-moment-error}
        \| \widehat{U}^{[2,2,1]}(t) - u(t) \|_{\ell_{h}^{2}}
        \leq
        C_T h^2,
        \qquad
        \| \widehat{P}^{[2,2,1]}(t) - p(t) \|_{\ell_{h}^{2}}
        \leq
        C_T h.
    \end{equation}
    for $t \in [0,T]$. Note that by the Cauchy--Schwarz inequality and the definition of the weighted norm in~\eqref{def:f-norm}, there exists a constant $C>0$ such that
    \begin{equation}
        \| j_f \|_{\ell_{h}^{2}} \leq C \| f \|_{X_h},
        \qquad
        \| \rho_f \|_{\ell_{h}^{2}} \leq C \| f \|_{X_h}
    \end{equation}
    for every $f \in X_h$. Applying these inequalities to
    \begin{equation}
        F = f_{h}^{n} - \hat{f}_{h}^{[2,2,1]}(t_n)
    \end{equation}
    and using Proposition~\ref{prop:population-error-refined}, we obtain
    \begin{equation}\label{eq:numerical-corrected-error}
        \| U_{h}^{n} - \widehat{U}_{h}^{[2,2,1]}(t_n) \|_{\ell_{h}^{2}}
        \leq
        C_T h^2,
        \qquad
        \| P_{h}^{n} - \widehat{P}_{h}^{[2,2,1]}(t_n) \|_{\ell_{h}^{2}}
        \leq
        C_T h
    \end{equation}
    for $0 \leq n \leq N_h$. Here, we used the fact that
    \begin{equation}
        \rho_F
        =
        \rho_{f_{h}^{n}} - \rho_{\hat{f}_{h}^{[2,2,1]}(t_n)},
        \qquad
        j_F
        =
        j_{f_{h}^{n}} - j_{\hat{f}_{h}^{[2,2,1]}(t_n)}.
    \end{equation}
    Combining~\eqref{eq:corrected-moment-error} and~\eqref{eq:numerical-corrected-error}, we conclude that
    \begin{equation}
        \| U_{h}^{n} - u(t_n) \|_{\ell_{h}^{2}}
        \leq
        C_T h^2,
        \qquad
        \| P_{h}^{n} - p(t_n) \|_{\ell_{h}^{2}}
        \leq
        C_T h
    \end{equation}
    for $0 \leq n \leq N_h$. This proves Theorem~\ref{thm:optimal_convergence}.
\end{proof}

\section{Concluding remarks}\label{sec:conclusion}

We have proved optimal convergence rates for the D2Q9 BGK lattice Boltzmann method with the half-way bounce-back rule for the incompressible Stokes equations in a flat channel. The convergence rates are second-order and first-order for the velocity and the pressure, in agreement with formal asymptotic analyses and numerical experiments. The main point of the analysis is the refinement of the prediction function constructed in previous works by adding Stokes correctors and Knudsen layers. The correctors are designed to absorb the leading boundary consistency errors of the previously used prediction function. The refined prediction function has a pointwise boundary consistency error of $O(h^5)$, while the interior consistency error is of sufficiently high order. Combined with the weighted $L^2$-stability estimate, these consistency estimates yield the optimal convergence rates.

Several extensions of the present result seem natural. The first is to consider non-half-way configurations. In this case, formal asymptotic analysis predicts first-order accuracy for the velocity~\cite{JunkYang2005}, whereas the convergence analysis of Junk and Yang~\cite{JunkYang2008,JunkYang2009} does not establish convergence for generic non-half-way configurations. The technique developed here suggests a possible route to prove the optimal first-order convergence for the velocity. In the non-half-way case, the leading-order terms in the asymptotic expansion of the boundary errors appear at a lower order than in the half-way case, and their macroscopic and kinetic components have to be identified to construct Stokes and Knudsen layer correctors.

Another direction is to extend our technique to more general collision operators and to the nonlinear Navier--Stokes setting. The convergence theorems in~\cite{JunkYang2008,JunkYang2009} include collision operators of general relaxation type and~\cite{JunkYang2009} considers the Navier--Stokes equations. For collision operators of general relaxation type, the construction of Stokes and Knudsen layer correctors would become more involved and require further analysis. For the nonlinear setting, it seems plausible that an extension of the present corrector construction combined with the perturbative argument in~\cite{JunkYang2009} could yield optimal convergence estimates for the nonlinear Navier--Stokes equations. This is left for future work.

Finally, extending the present analysis to domains with curved boundaries is also of great importance. We expect that the construction of the Knudsen layer correctors needs to incorporate the curvature of the boundaries as in the higher-order Knudsen layers in the kinetic theory of gases; see for example~\cite{Sone2002,HattoriTakata2015b}. Thus understanding the effect of boundary geometry on the bounce-back rule is an interesting future direction.

\backmatter

\bmhead{Acknowledgements}
The author has been supported by JSPS KAKENHI Grant Number 25K07077.

\section*{Statements and Declarations}

\noindent \textbf{Competing interests.}
The author declares no competing interests.

\noindent \textbf{Data availability.}
No datasets were generated or analyzed during the current study.

\noindent \textbf{Use of generative AI.}
During the development of this work, the author used OpenAI's ChatGPT as an interactive research tool for exploring and refining proof strategies, checking intermediate arguments and calculations, assisting with literature searches, and drafting or revising parts of the exposition. Mathematical content arising from these interactions was critically examined, independently verified, and revised by the author before inclusion in the manuscript. The author takes full responsibility for all mathematical claims and for the final form of the manuscript.

\bibliography{main.bib}

\end{document}